%% file: urmat.tex
\documentclass[11pt]{article} 

\usepackage[utf8]{inputenc} 

\usepackage{cancel}

\usepackage{geometry} 
\usepackage{graphicx} 
\usepackage[english]{babel}

\usepackage{booktabs} 
\usepackage{array} 
\usepackage{paralist} 
\usepackage{verbatim} 
\usepackage{subfig} 

\usepackage{amssymb}
\usepackage{bbm}
\usepackage{amsmath}
\usepackage{amsthm}
\usepackage{makecell}

\usepackage{epigraph}

\usepackage{hyperref}
\hypersetup{
    colorlinks=true,
    citecolor=blue
}

\usepackage{pdflscape}
\usepackage{natbib}
\usepackage{wrapfig}
\usepackage{subfig}
\usepackage[capitalize,noabbrev]{cleveref}
\usepackage{authblk}
\usepackage{changepage}

\usepackage{tikz-cd}
\usepackage{xcolor}

\usepackage{algorithm}
\usepackage{algpseudocode}

\input{macros}

\title{Positive Semidefinite Matrix Supermartingales\footnote{Accepted for publication in the \emph{Electronic Journal of Probability}}
}

\author[1]{Hongjian Wang}
\author[2]{Aaditya Ramdas}
\affil[1, 2]{Department of Statistics and Data Science, Carnegie Mellon University}
\affil[2]{Machine Learning Department, Carnegie Mellon University} 
\affil[ ]{\texttt{ \{hjnwang,aramdas\}@cmu.edu  }}

\date{\today}

\newtheorem{theorem}{Theorem}[section]
\newtheorem{definition}[theorem]{Definition}

\newtheorem{fact}{Fact}
\newtheorem{proposition}[theorem]{Proposition}
\newtheorem{hypothesis}[theorem]{Conjecture}

\newtheorem{corollary}[theorem]{Corollary}
\newtheorem{lemma}[theorem]{Lemma}
\theoremstyle{definition}\newtheorem{remark}[theorem]{Remark}
\newtheorem{example}[theorem]{Example}

\Crefname{fact}{Fact}{Facts}

\begin{document}

\maketitle

\begin{abstract}
    \input{urmat/abstract}
\end{abstract}

\input{urmat/maintext-v2}

\input{urmat/appendices}

\input{urmat/conclusion}

\subsubsection*{Acknowledgement}

\input{urmat/ack}

\bibliography{urmat}




\end{document}

%% file: macros.tex
\renewcommand{\le}{\leqslant}

\renewcommand{\ge}{\geqslant}

\newcommand{\matle}{\preceq}
\newcommand{\matge}{\succeq}
\newcommand{\nmatle}{\npreceq}
\newcommand{\nmatge}{\nsucceq}
\newcommand{\matls}{\prec}
\newcommand{\matgt}{\succ}

\newcommand{\id}{\mathbbmss 1}
\newcommand {\trsp}{^\mathsf {T}}
\newcommand {\matl}{\left[ \begin{matrix}}
\newcommand {\matr}{\end{matrix}\right]}
\newcommand {\Exp}{ \mathbb E }

\newcommand {\sg}{\operatorname{sgn}}
\renewcommand {\Pr}{ \mathbb P }

\newcommand{\mtx}[1]{\mathbf{#1}}
\newcommand{\vct}[1]{\mathbf{#1}}

\newcommand{\mX}{ \mtx{X} }
\newcommand{\mY}{ \mtx{Y} }
\newcommand{\mZ}{ \mtx{Z} }
\newcommand{\mV}{ \mtx{V} }
\newcommand{\mI}{ \mtx{I} }
\newcommand{\mM}{ \mtx{M} }
\newcommand{\mH}{ \mtx{H} }
\newcommand{\mU}{ \mtx{U} }
\newcommand{\mG}{ \mtx{G} }
\newcommand{\mSg}{ \mtx{\Sigma} }

\newcommand{\mA}{ \mtx{A} }
\newcommand{\mB}{ \mtx{B} }
\newcommand{\mC}{ \mtx{C} }

\newcommand{\mP}{ \mtx{P} }
\newcommand{\mE}{ \mtx{E} }
\newcommand{\mF}{ \mtx{F} }

\newcommand{\mL}{ \mtx{L} }

\newcommand{\mR}{ \mtx{R} }

\newcommand{\uzeroone}{\operatorname{Unif}_{(0,1)}}

\newcommand{\sdiv}{a}
\newcommand{\mdiv}{\mtx{A}}

\newcommand{\vv}{\vct{v}}
\newcommand{\vu}{\vct{u}}
\newcommand{\vx}{\vct{x}}

\newcommand {\Var}{\mathbf{Var}}
\newcommand {\Cov}{\mathbf{Cov}}

\newcommand{\e}{\mathrm e}

\newcommand{\cA}{\mathcal{A}}
\newcommand{\cX}{\mathcal{X}}
\newcommand{\cY}{\mathcal{Y}}
\newcommand{\cC}{\mathcal{C}}

\newcommand{\cF}{\mathcal{F}}
\newcommand{\cS}{\mathcal{S}}
\newcommand{\cG}{\mathcal{G}}
\newcommand{\cB}{\mathcal{B}}
\newcommand{\cN}{\mathcal{N}}

\newcommand{\cE}{\mathcal{E}}

\newcommand{\cM}{\mathcal{M}}
\newcommand{\cH}{\mathcal{H}}
\newcommand{\cP}{\mathcal{P}}

\newcommand{\iid}{\overset{\mathrm{iid}}{\sim}}

\newcommand{\psiE}{\psi_{\mathrm{E}}}
\DeclareMathAlphabet{\mathbbmsl}{U}{bbm}{m}{sl}

\newcommand{\tr}{\mathrm{tr}}

\newcommand{\abs}{\operatorname{abs}}

\newcommand{\tm}{{\textsc{Matrix}}}
\newcommand{\ts}{{\textsc{Scalar}}}

\newcommand{\emtm}{\emph{\tm}}
\newcommand{\emts}{\emph{\ts}}
\newcommand{\er}{\operatorname{r}}

%% file: urmat/abstract.tex
We explore the asymptotic convergence and nonasymptotic maximal inequalities of supermartingales and backward submartingales in the space of positive semidefinite matrices. These are natural matrix analogs of scalar nonnegative supermartingales and backward nonnegative submartingales, whose convergence and maximal inequalities are the theoretical foundations for a wide and ever-growing body of results in statistics, econometrics, and theoretical computer science.

Our results lead to new concentration inequalities for either martingale dependent or exchangeable random symmetric matrices under a variety of tail conditions, encompassing now-standard Chernoff bounds to self-normalized heavy-tailed settings. Further, these inequalities are usually expressed in the Loewner order, are sometimes valid simultaneously for all sample sizes or at an arbitrary data-dependent stopping time, and can often be tightened via an external randomization factor.
    

%% file: urmat/maintext-v2.tex
\section{Introduction}

Let $\mY_0, \mY_1,\dots$ be a sequence of random symmetric $d\times d$ matrices adapted to a filtration $\{\cF_n
\}_{n\ge 0 }$. The process $\{ 
\mY_n \}_{n\ge 0 }$ is called a \emph{(matrix-valued) supermartingale} if
\begin{equation}\label{eqn:mat-sm}
    \text{for all }n \ge 1, \quad \Exp(\mY_n | \cF_{n-1}) \matle \mY_{n-1}. 
\end{equation}
Here, crucially, the order $\matle$ is the \emph{Loewner order}, where $\mA \matle \mB$ if and only if $\mB - \mA$ is a positive semidefinite matrix. This introduces the rich spectral theory of matrices to the familiar $d=1$ scalar special case as well as the matrix martingale special case where the ``$\matle$'' in \eqref{eqn:mat-sm} takes the entry-wise ``$=$''.

Similarly, if the matrices $\{ \mY_n \}_{n \ge 0}$ are adapted instead to a \emph{backward} filtration $\{\cG_n
\}_{n \ge 0}$, that is, $\cG_0 \supseteq \cG_1 \supseteq \dots$, it is called a \emph{(matrix-valued) backward submartingale} if
\begin{equation}\label{eqn:mat-bsm}
    \text{for all }n \ge 1, \quad \Exp(\mY_{n-1} | \cG_{n}) \matge \mY_{n}. 
\end{equation}
In this paper, we primarily study matrix-valued supermartingales and backward submartingales that are almost surely \emph{positive semidefinite} (PSD). These matrix-valued stochastic processes are natural and non-trivial generalizations of scalar nonnegative supermartingales and nonnegative backward submartingales, studied extensively by the probability and statistics literature. In particular, as we shall review soon in \cref{sec:scalars}, a scalar nonnegative supermartingale or backward submartingale $\{ Y_n \}$ satisfies \emph{Ville's inequality}
\begin{equation}
   \Pr( \exists n, \ Y_n \ge \sdiv ) \le (\Exp Y_0) \sdiv ^{-1},
\end{equation}
which is interconnected with various results on the asymptotic and non-asymptotic behavior of these classes of processes, and has led to numerous methods in sequential statistics. We extend all these results and applications from scalars to matrices. In particular, we present the matrix Ville's inequality
\begin{equation}
        \Pr( \exists n, \; \mY_n \nmatle \mA) \le \tr( (\Exp \mY_0) \mA^{-1} )
\end{equation}
for both classes of processes, their numerous equivalent or strengthened forms, and a variety of statistical applications.

A list of major contributions of this paper is as follows.
\begin{enumerate}
    \item The convergence and optional stopping theorems, and the Ville's inequality for PSD matrix supermartingales  (Theorems~\ref{thm:matdoobconv}, \ref{thm:matrix-optional-stopping}, and~\ref{thm:matvil1}) that generalize the corresponding scalar results.
    \item Sequential testing procedure with a matrix test supermartingale (\cref{alg:seqtest}) that can be applied to the meaning testing of bounded (Examples~\ref{ex:bet} and~\ref{ex:emp-bern-ms}), symmetric (Example~\ref{ex:sym}), or heavy-tailed (Example~\ref{ex:sn} and~\ref{ex:cat}) martingale-dependent matrices.
    \item Matrix e-values and e-processes (Definitions~\ref{def:mat-ev} and~\ref{def:mat-epr}), generalizing the corresponding scalar concepts that have recently been extensively studied.
    \item The matrix Ville's inequality for backward PSD matrix submartingales (Theorem~\ref{thm:matbwvil}), leading to the exchangeable matrix Chebyshev's inequality (Theorem~\ref{thm:xmci}). 
    \item The randomized matrix Markov's and Ville's inequalities (Theorems~\ref{thm:ur-matmarkov-meta} and~\ref{thm:matvil1-u}).
    \item A sequential mean testing example with the matrix supermartingale outperforming the scalar supermartingale (Example~\ref{ex:noncommute}).
\end{enumerate}

The rest of the paper is organized as follows. We review in \cref{sec:prelim} the technical background of matrix analysis, sequential statistics with scalar processes in the abovementioned classes, as well as matrix concentration inequalities in the literature. We study forward PSD matrix supermartingales, their Ville's inequality and applications in \cref{sec:mtg}. We then study backward PSD matrix submartingales, their Ville's inequality and applications in \cref{sec:bw}. In \cref{sec:dis} we provide several perspectives comparing our inherent matrix approach to some existing scalar-born or scalarized results, and discuss a randomized improvement that is further developed in \cref{sec:ummi-full}. An interesting empirical study is presented in \cref{sec:exp}, including a problem where our matrix method yields higher out-of-the-box power compared to the previous scalar method. Omitted proofs are found in \cref{sec:pf}; some further references and independent theoretical results in \cref{sec:misc}.

\section{Preliminaries and Related Work}\label{sec:prelim}

\subsection{Matrix Analysis}

Let $\cS_d$ denote the set of all $d\times d$ real-valued symmetric matrices, which is the only class of matrices considered in this paper.\footnote{We only briefly discuss the case with non-symmetric and non-square matrices in \cref{sec:dilation}.} These matrices are denoted by bold upper-case letters $\mA, \mB$, etc; whereas $d$ dimensional (column) vectors by bold lower-case letters $\vu, \vv$, etc. For $I \subseteq \mathbb R$, we denote by $\cS_d^I$ the set of all real symmetric matrices whose eigenvalues are all in $I$.  $\cS_d^{ [0,\infty) }$, the set of positive semidefinite and $\cS_d^{ (0,\infty) }$, the set of positive definite matrices are simply denoted by $\cS_d^{+}$ and $\cS_d^{++}$ respectively. We consider random matrices taking values in $\cS_d$. It is easy to see that if $\mX$ is an integrable random matrix taking values in $\cS_d^I$ where $I$ is an interval, any conditional expectation $\Exp(\mX | \cA)$ also takes values in $\cS_d^I$.
The Loewner partial order $\mA \matle \mB$ means $\mB - \mA$ is positive semidefinite, and $\mA \matls \mB$ means $\mB - \mA$ is positive definite.

Any function $f: I \to J$ (where $I, J \subseteq \mathbb R$) is identified canonically with a function $f: \cS_d^I \to \cS_d^J$ in the following way: for a diagonal matrix $\mSg \in \cS_d^I$, $f(\mSg)$ is obtained by applying $f$ entrywise; otherwise it is obtained by applying $f$ after orthogonal diagonalization, $f(\mU \trsp \mSg \mU) = \mU\trsp f(\mSg) \mU$. Notable matrix-to-matrix operators that recur in this paper defined this way include the matrix exponential $\exp: \cS_d \to \cS^{++}_d$, the matrix logarithm $\log: \cS^{++}_d \to \cS_d $,  the matrix absolute value\footnote{We use $\abs(\, \cdot \,)$ instead of $| \cdot |$ to distinguish it from the determinant and other norms.} $\abs : \cS_d \to \cS^+_d$, and powers $\mX \mapsto \mX^k$ for any $k \in \mathbb R$. 

Important functionals $\cS_d \to \mathbb R$ we use throughout the paper include $\lambda_{\max}$, the largest eigenvalue of a matrix; $\| \cdot \|$, the spectral norm of a matrix which equals $\lambda_{\max} \circ \abs$; and $\tr$, the trace. Note that $\tr$ preserves the orders (viz., $\mA \matle \mB$ implies $\tr \mA \le \tr \mB$) and is lower bounded by $\lambda_{\max} = \| \cdot \|$ on $\cS_d^{+}$.

It is rare to find matrix-to-matrix functions that are monotone (i.e.\ preserving the $\matle$ order) and convex (i.e.\ $f$ such that $f(\sum \lambda_i \mX_i) \matle \sum {\lambda_i} f(\mX_i)$ if $\sum \lambda_i = 1$ and each $\lambda_i \ge 0$; concavity is defined similarly). Scalar-to-scalar monotonicity ($I \to J$)  does \emph{not} imply matrix-to-matrix monotonicity  ($\cS_d^I \to \cS_d^J$), and the same with convexity. For example, the matrix square $\mX \mapsto \mX^2$ is not monotone on the domain $\cS_d^+$ and the matrix 4\textsuperscript{th} power $\mX \mapsto \mX^4$ is neither monotone nor convex. The following are some of the few well-known nontrivial ``operator monotone" and ``operator convex" functions in the matrix analysis literature (cf. \citet[Chapter V]{bhatia2013matrix}). 
\begin{fact}[Operator Monotone and Convex Functions]\label{fct:opmonofuns}
   The matrix logarithm $\mX \mapsto \log \mX$ ($\cS_d^{++} \to \cS_d$) is monotone and concave. For $k \in (0,1)$, the matrix power $\mX \mapsto \mX^k$ ($\cS_d^+ \to \cS_d^+$) is monotone and concave. For $k \in (1,2]$, the matrix power $\mX \mapsto \mX^k$ ($\cS_d^+ \to \cS_d^+$ if $k \neq 2$; $\cS_d \to \cS_d^+$ if $k = 2$) is convex.
\end{fact}
With trace, however, real-to-real functions translate into matrix-to-real functions with monotonicity and convexity intact. That is,
\begin{fact}[Trace Monotonicity]\label{fct:tracemono}
    Let $f:I \to \mathbb R$ be monotone increasing ($I \subseteq \mathbb R$). Then so is $\tr \circ f : \cS_d^I \to \mathbb R$, i.e.,
\begin{equation}
\mA \matle \mB \implies \tr f(\mA) \le \tr f(\mB).
\end{equation}
\end{fact}

\begin{fact}[Trace Jensen]\label{fct:tracejen} For an interval $I \subseteq \mathbb R$, assume $f: I \to \mathbb R$ is convex. Then so is $\tr \circ f : \cS_d^I \to \mathbb R$, i.e.,
\begin{equation}
    \tr f\left(\sum \lambda_i \mX_i \right) \le \tr \left( \sum \lambda_i f(\mX_i) \right),
\end{equation}
    if $\sum \lambda_i = 1$ and each $\lambda_i \ge 0$.
\end{fact}

\subsection{Scalar Nonnegative Supermartingales and Statistical Applications}\label{sec:scalars}

A scalar-valued nonnegative supermartingale $\{ Y_n \}_{n \ge 0}$ on the filtration $\{ \cF_n \}_{n \ge 0}$ satisfies Ville's inequality \citep{ville1939etude},
\begin{equation}\tag{VI}\label{eqn:vi}
    \Pr( \exists n, \ Y_n \ge \sdiv ) \le (\Exp Y_0) \sdiv ^{-1},
\end{equation}
or equivalently,
\begin{equation}\tag{VI}\label{eqn:vi-stop}
    \Pr( Y_\tau \ge \sdiv  ) \le (\Exp Y_0) \sdiv^{-1}\quad \text{for any stopping time }\tau.
\end{equation}
The equivalence is demonstrated by, for example, \citet[Lemma 3]{howard2021time}.
Ville's inequality is often seen as a time-uniform, or a ``stopped'' generalization of
Markov's inequality 
\begin{equation}\label{eqn:mi}\tag{MI}
    \Pr( X \ge \sdiv  ) \le (\Exp X) \sdiv ^{-1}
\end{equation}
for a nonnegative random variable $X$.

As Markov's inequality lays the foundation for a large body of concentration inequalities for a fixed number of random variables leading up to non-asymptotic statistical methods (e.g.\ confidence intervals) valid at fixed finite sample sizes,
 Ville's inequality \eqref{eqn:vi} has been the workhorse for many recent time-uniform concentration inequalities and accordingly \emph{sequential} statistical methods which are valid at data-dependent sample sizes. {These methods enable conducting hypothesis testing and estimation with uncertainty quantification without being committed to an \emph{a priori} fixed sample size, allowing the experimenter to stop the experiment for any (possibly data-dependent) reason --- formally, at a \emph{stopping time} $\tau$ on the filtration $\cF_n = \sigma(X_1,\dots,X_n)$ generated by the data sequence $X_1,X_2,\dots$.
 \begin{itemize}
     \item  A $(1-\alpha)$-confidence \emph{sequence} for the parameter $\theta$ is a sequence of sets $\{ C_n \}_{n \ge 1}$ (each $C_n$ computed from $X_1,\dots,X_n$) such that $\Pr( \theta \in C_n \text{ for all }n  ) \ge 1-\alpha$, or equivalently $\Pr( \theta \in C_\tau  ) \ge 1-\alpha$ for any stopping time $\tau$. They are usually derived via equation \eqref{eqn:vi} \citep{howard2021time,waudby2020estimating,waudby2020confidence,wang2023catoni,wang2023huber}.
     \item When testing the null hypothesis that the true data-generating distribution belongs to the set $\cH_0$, a \emph{test supermartingale} for $\cH_0$ is a process $\{ Y_n \}$ (each $Y_n$ computed from $X_1,\dots,X_n$) with $Y_0 = 1$ that is a nonnegative supermartingale, i.e.\ $\Exp_{P}(Y_n|X_1,\dots, X_{n-1})\le Y_{n-1}$, under \emph{any} distribution $P \in \cH_0$. One rejects $\cH_0$ when the observed $Y_n$ first surpasses $1/\alpha$. This ensures, due to \eqref{eqn:vi-stop}, that the type-I error rate of \emph{ever} falsely rejecting $\cH_0$ is at most $\alpha$.
 \end{itemize} 
 These \eqref{eqn:vi}-driven sequential methods are widely applied in real-world cases where it is much more economical (see e.g.\ the argument made by \cite{woong2023design}) to leave the sample size unspecified and allow flexible sequential experiments.
See the recent review article by \citet[Sections 4 and 5]{ramdas2022game} for numerous examples. 
 We revisit this topic with a more involved discussion on \emph{test processes} in \cref{sec:mat-e-proc}.}

 Numerous variations and extensions on \eqref{eqn:vi} exist. First, the backward Ville's inequality (see e.g.\ \citet[Theorem 2]{manole2023martingale}). A process $\{ Y_n \}$ adapted to a backward filtration $\{ \cG_n \}$ is called a backward \emph{sub}martingale if $\Exp( Y_{n-1}|\cG_n) \ge Y_n$ which, if nonnegative, satisfies \eqref{eqn:vi} as well,
 \begin{equation}\label{eqn:bvi}
      \Pr( \exists n, \ Y_n \ge \sdiv ) \le (\Exp Y_0) \sdiv ^{-1}.
 \end{equation}
 Second, the uniformly randomized Ville's inequality \citep[Theorem 4.1]{ramdas2023randomized}, which states that \eqref{eqn:vi-stop} can be tightened into
 \begin{equation}\label{eqn:uvi}
    \Pr( Y_\tau \ge U \sdiv  ) \le (\Exp Y_0) \sdiv^{-1}\quad \text{for any stopping time }\tau,
\end{equation}
where $U \sim \operatorname{Unif}_{(0,1)}$ is independent from the filtration $\{\cF_n\}$. \citet[Section 4.3]{ramdas2023randomized} outline its implication for sequential statistics, which we discuss in \cref{sec:randfix} in the context of matrices.



\subsection{Matrix Concentration Inequalities}
\subsubsection{Spectral Bounds}\label{sec:spectral-bounds}

There has been a wealth of results bounding the tails of the largest eigenvalue (or the spectral norm) of the sum of independent or martingale dependent random symmetric matrices. Most notably, \cite{oliveira2009concentration,oliveira2010sums}, \cite{tropp2011user,tropp2012user,tropp2015introduction}, and their collaborators achieve the technically involved extension of the Cram\'er-Chernoff method, the central tool in deriving exponentially decaying scalar concentration inequalities, to random matrices satisfying similar tail conditions. The most prominent step along this route is the application of Lieb's concavity theorem \citep{lieb1973convex} that ends up bounding the expected value of the \emph{scalar}
\begin{equation}\label{eqn:traceproc}
   \tr \exp\left( \sum \mX_i \right)
\end{equation}
where $\mX_i$'s are independent \citep[Section 3.4]{tropp2012user} or martingale dependent \citep[Section 2.2]{tropp2011user} matrices. Fixed-time tail bounds follow from Markov's inequality \eqref{eqn:mi}. Already noticed by \cite{tropp2011user} and further exploited by \cite{howard2020time}, the trace process \eqref{eqn:traceproc} is a (scalar) nonnegative supermartingale, and Ville's inequality \eqref{eqn:vi} thus implies time-uniform spectral tail bounds for the same classes of random matrices.

Our new results will take a different approach as we shall construct matrix-valued supermartingales and Loewner-ordered ($\matle$) concentration inequalities. Their fixed-time counterparts are reviewed next in \cref{sec:aw}. Our results can often recover these ``scalarized'' results of \citeauthor{oliveira2009concentration,tropp2011user}, and \citeauthor{howard2020time} mentioned above.

\subsubsection{Ahlswede-Winter Bounds}\label{sec:aw}
\citet[Theorem 12]{ahlswede2002strong} prove the following Loewner-ordered matrix Markov's inequality
\begin{equation}\label{eqn:mmi}\tag{MMI}
         \Pr( \mX  \nmatle  \mdiv) \le \tr( (\Exp \mX) \mdiv^{-1}),
\end{equation}
where $\mX$ is an $\cS_d^+$-valued random matrix and $\mA \in \cS_d^{++}$ is deterministic. 
This leads to the direct corollary of a matrix Chebyshev inequality that serves as a weak law of large numbers for the average $\overline{\mX}_n$ of $n$ i.i.d.\ symmetric matrices with mean $\mM$ and (matrix) variance $\mV =  \Exp (\mX - \mM)^2$ \citep[Corollary 16]{ahlswede2002strong},
\begin{equation}\label{eqn:mci}
     \Pr( \abs(\overline{\mX}_n - \mM) \nmatle \mdiv ) \le n^{-1} \tr( \mV \mdiv^{-2} ).
\end{equation}
\cite{ahlswede2002strong} also attempt to use this ``un-scalarized'' method to generalize the Cram\'er-Chernoff method for scalars. However, as summarized by \cite[Section 3.7]{tropp2012user}, this leads to suboptimal bounds that involve ``the sum of eigenvalues'' as opposed to ``the eigenvalue of a sum'' which the sharper bounds (obtained by ``scalarizing'' with Lieb's concavity theorem as mentioned in \cref{sec:spectral-bounds}) involve.

Our time-uniform, martingale-based main results are often obtained via \eqref{eqn:mmi}. In particular, we show that this Loewner-ordered approach pioneered by \cite{ahlswede2002strong} does lead to sharp matrix concentration inequalities and tests, fixed-time and time-uniform alike, comparable to and sometimes better than the results cited in \cref{sec:spectral-bounds}. We provide a self-contained proof of \eqref{eqn:mmi} in \cref{sec:ummi-full} that also comes with the extension of randomization.

\section{Forward PSD Supermartingales}\label{sec:mtg}

We consider in this section a probability space with a usual forward filtration $\{\cF_n \}_{n \ge 0}$. That is, $\cF_0 \subseteq \cF_1 \subseteq \dots \subseteq \cF_{\infty}$. Within the scope of this section, we shall often omit the filtration $\{ \cF_n \}$ when unambiguous from the context, calling a scalar- or matrix-valued process ``adapted'' if the $n$-indexed element is $\cF_n$-measurable; ``predictable'' if  $\cF_{n-1}$-measurable. Stopping times are all with respect to $\{\cF_n\}$.


\subsection{Convergence, Optional Stopping, and Maximal Inequalities}

Recall our definition of a matrix-valued supermartingale from the beginning of the paper that
an adapted $\cS_d$-valued stochastic process  $\{ \mY_n \}$
is a supermartingale if $\Exp (\mY_n | \cF_{n-1} ) \matle \mY_{n-1}$ holds for every $n$. Before we proceed, note that the order $\matle$ is preserved by addition, limit, and Lebesgue integration (hence including convergent infinite sums, expectations, conditional expectations, etc., essentially because $\cS^{+} = \{ \mA : \mA \matge 0 \}$ is closed both topologically and under $+$). See also a ``dominated integrability" lemma we prove in \cref{sec:domint}.

Matrix supermartingales are related to scalar supermartingales in the following ways. First, the monotonicity of trace implies that if $\{ \mY_n \}$ is an $\cS_d$-valued supermartingale, its trace process $\{ \tr (\mY_n) \}$ is a real-valued supermartingale. Second, we have the following lemma {($\mathbb Q$ being the set of rational numbers)}. 

\begin{lemma}\label{lem:nsmiff}
    (1) An $\cS_d$-valued adapted process $\{ \mY_n \}$ is a supermartingale if and only if for every non-random vector $\vv \in \mathbb R^d$ {or $\vv \in \mathbb Q^d$}, the scalar-valued process $\{ \vv\trsp \mY_n \vv \}$ is a supermartingale.  {(2) Similar to the scalar case, an $\cS_d$-valued integrable process $\{\mY_n\}$ is a supermartingale if and only if the predictable component of its Doob decomposition $\mA_n = \sum_{i=1}^n  (\Exp(\mY_i|\cF_{i-1}) - \mY_{i-1})$ is decreasing in the $\matle$ order.}
\end{lemma}
The
proof of Lemma~\ref{lem:nsmiff} can be found in \cref{sec:pf-tp}. It leads to an analog of Doob's martingale convergence theorem for PSD matrix supermartingales.

\begin{theorem}[Matrix Supermartingale Convergence Theorem]\label{thm:matdoobconv}
    Any $\cS_d^+$-valued supermartingale $\{ \mY_n \}$ converges almost surely to an $\cS_d^+$-valued random matrix $\mY_\infty$ and $\Exp \mY_\infty \matle \Exp \mY_0$.
\end{theorem}
\begin{proof}
For any $\vv \in \mathbb R^d$, the process $\{ \vv\trsp \mY_n \vv \}$ is a scalar-valued nonnegative supermartingale, which converges a.s.\ to an integrable random scalar due to Doob's martingale convergence theorem (see e.g.\ Corollary 11.5 in \cite{klenke2013probability}). First, setting $\vv$ to $(0, \dots, 0, 1, 0 ,\dots, 0)\trsp$ we see that $\mY_n$'s diagonal converges to integrable scalars. Then, setting $\vv$ to $(0, \dots, 0, 1, 0 ,\dots, 0, 1, 0 ,\dots, 0)\trsp$ we see that $\mY_n$'s off-diagonal entries also converge to integrable scalars. The limit $\mY_\infty = \lim_{n\to \infty} \mY_n$ is in $\cS_d^+$ because $\cS_d^+$ is a closed set. Finally, for any $\vv \in \mathbb R^d$,
\begin{equation}
   \vv\trsp (\Exp \mY_\infty ) \vv  = \lim_{n \to \infty } \Exp( \vv \trsp \mY_n \vv ) \le \Exp( \vv \trsp \mY_0 \vv ) =   \vv \trsp (\Exp \mY_0) \vv,
\end{equation}
so $\Exp \mY_\infty \matle \Exp \mY_0$.
\end{proof}

The convergence of PSD supermartingales ensures that we can speak of the stopped matrix $\mY_\tau$ for any stopping time $\tau$, even when $\Pr(\tau = \infty) > 0$. We can use a similar method to extend some scalar versions of the optional stopping (sampling) theorems to matrices. For example:

\begin{theorem}[Matrix Optional Stopping]\label{thm:matrix-optional-stopping}
    Let $\{ \mY_n \}$ be an $\cS_d$-valued supermartingale, and {$\sigma \le \tau$ be two} stopping times. Then:
    \begin{enumerate}
        \item $\{ \mY_{n \wedge \tau} \}$ is also a supermartingale {on both $\{ \cF_n \}$ and $\{ \cF_{n\wedge \tau} \}$}.
        \item {If  (a) $\tau < N$ for some finite $N$, or (b) $\tau < \infty$ and $\{\mY_n\}$ is $\cS_d^{+}$-valued,
       then $ \mathbb E(\mY_\tau | \mathcal {F}_{\sigma}) \matle \mY_\sigma$. Consequently, $\Exp \mY_\tau \matle \Exp \mY_\sigma \matle \Exp \mY_0$.}
    \end{enumerate}
\end{theorem}
\begin{proof}{The proofs of two statements are as follows.
 \begin{enumerate}
        \item For any $\vv \in \mathbb R^d$, the scalar process $\{ \vv \trsp \mY_n \vv \}$ is a supermartingale on $\{ \cF_n \}$ and therefore so is the stopped process $\{ \vv \trsp \mY_{n\wedge \tau} \vv \}$ on both $\{ \cF_n \}$ and $\{ \cF_{n\wedge \tau} \}$ due to the optional stopping theorem (e.g.\ Theorem 10.15 of \citet{klenke2013probability}). This implies that $\{  \mY_{n \wedge \tau} \}$ is  a matrix supermartingale on both $\{ \cF_n \}$ and $\{ \cF_{n\wedge \tau} \}$.
        \item  For any $\vv \in \mathbb R^d$, we apply the optional sampling theorem (e.g.\ Theorem 10.11(i,ii) of \citet{klenke2013probability}) to the scalar-valued supermartingale $\{ \vv \trsp \mY_n \vv \}$ to obtain that $\vv\trsp \mY_\sigma \vv \ge  \Exp(\vv\trsp \mY_\tau \vv | \cF_\sigma) = \vv\trsp \Exp( \mY_\tau | \cF_\sigma)  \vv$. Therefore  $\mY_\sigma \matge \mathbb E(\mY_\tau | \mathcal {F}_{\sigma})$.
 \qedhere
    \end{enumerate}}
\end{proof}

This gives rise to the matrix versions of the two mutually equivalent statements of Ville's inequality \eqref{eqn:vi} via some manipulation of crossing events.

\begin{theorem}[Matrix Ville's Inequality]\label{thm:matvil1}
    Let $\{ \mY_n \}$ be an $\cS_d^+$-valued supermartingale. 
    Then, for any $\mA \in \cS_d^{++}$,
    \begin{equation}\tag{MVI}\label{eqn:mvi-new}
        \Pr( \exists n, \; \mY_n \nmatle \mA) \le \tr( (\Exp \mY_0) \mA^{-1} );
    \end{equation}
    or equivalently,
     \begin{equation}\tag{MVI}\label{eqn:mvi-stopped}
     \text{for any stopping time }\tau, \quad   \Pr( \mY_\tau \nmatle \mA) \le \tr( (\Exp \mY_0) \mA^{-1} ).
    \end{equation}
\end{theorem}
\begin{proof}
   We prove the stopped inequality first. Due to Part 3 of Theorem~\ref{thm:matrix-optional-stopping}, $\Exp \mY_{\tau} \matle \Exp \mY_0$. Now by \eqref{eqn:mmi} and the {monotonicity} of the trace,
    \begin{equation}
        \Pr( \mY_{\tau} \nmatle  \mA ) \le \tr(  \mA^{-1/2} (\Exp \mY_{\tau }) \mA^{-1/2} ) \le  \tr(  \mA^{-1/2} (\Exp \mY_0) \mA^{-1/2} ),
    \end{equation}
   which is the stopped \eqref{eqn:mvi-stopped}.  Next, we define the stopping time $\nu := \inf\{ n: \mY_n \nmatle \mA  \}$. Applying the result above with $\tau = \nu \wedge N$,
    \begin{equation}
      \Pr(\nu \le N) =  \Pr( \mY_{\nu \wedge N} \nmatle \mA) \le \tr(  (\Exp \mY_0) \mA^{-1} )
    \end{equation}
    for any $N \in \mathbb N$. Letting $N \to \infty$, we have
    \begin{equation}
        \Pr(\nu < \infty)\le \tr(  (\Exp \mY_0) \mA^{-1} )
    \end{equation}
    which concludes the proof of both forms of \eqref{eqn:mvi-new}.
\end{proof}

A natural question arises: in comparison to the scalar \eqref{eqn:vi}, does \eqref{eqn:mvi-new} provide tighter, looser, or incomparable concentration results? 
This question is essential as from every PSD supermartingale $\{\mY_n\}$ one can construct scalar nonnegative supermartingales $\{\tr \mY_n \}$ and $\{ 
\vv \trsp \mY_n \vv \}$.
We shall answer the question in detail, theoretically in \cref{sec:tight,sec:sclr} and empirically in \cref{sec:exp}. As our exposition unfolds, the following three conclusions shall become clear:
\begin{enumerate}
    \item \eqref{eqn:mvi-new} is in general not reducible to the scalar \eqref{eqn:vi}.
    \item To conceptualize the difference between \eqref{eqn:mvi-new} and \eqref{eqn:vi}, it suffices to look at the one-time case, comparing \eqref{eqn:mmi} and \eqref{eqn:mi}. The spectra of both the ``threshold'' matrix $\mA$ and the random matrix $\mX$ matter.
    \item While the one-time Markov case reveals the potential tightness benefit of \eqref{eqn:mvi-new}, in real data the conditions for such benefit often arise naturally in the sequential, anytime-valid testing case that necessitates Ville-type inequalities. 
\end{enumerate}
As an aside, we present the additional result of a matrix analog of Doob's $L^p$ inequality for forward matrix martingales in \cref{sec:further}. It is not necessarily relevant to our upcoming discussion in statistical contexts but might be of independent interest.

\subsection{Matrix Test Supermartingales}\label{sec:testproc}

Scalar-valued nonnegative supermartingales are often constructed incrementally: if one can find an adapted nonnegative process $\{ E_n \}$ that has the conditional expectation bound $\Exp(E_n|\cF_n) \le 1$, then their product $Y_n = \prod_{i=1}^n E_i$ is a nonnegative supermartingale {which satisfies $Y_0 = 1$ and thus satisfies \eqref{eqn:vi} with $\Pr(Y_\tau \ge \sdiv) \le \sdiv^{-1}$.}

We propose the following analogous way to construct $\cS_d^+$-valued supermartingales that involves a bit more flexibility.
\begin{lemma}[Incremental Construction of Matrix Supermartingales]\label{lem:mateval}
If an adapted $\cS_{d}^+$-valued process $\{\mE_n\}_{n \ge 1}$ and a predictable $\cS_{d}^{++}$-valued process $\{\mF_n\}_{n \ge 1}$ satisfy $\Exp(\mE_n | \cF_{n-1}) \matle \mF_n^{-1}$ almost surely for all $n \ge 1$.  Then, the following is a PSD supermartingale:
    \begin{equation}\label{eqn:e-to-nsm}
        \mY_n = \sqrt{\mF_{1}} \sqrt{\mE_{1}} \sqrt{\mF_{2}} \sqrt{\mE_{2}}\dots  \sqrt{\mF_{n}} \sqrt{\mE_{n}}  \sqrt{\mE_n}  \sqrt{\mF_{n}} \dots \sqrt{\mE_{2}}  \sqrt{\mF_{2}} \sqrt{\mE_{1}}  \sqrt{\mF_{1}},
    \end{equation}
    with $\mY_0 = \mI$.
\end{lemma}
In some cases, we can take $\mF_1 = \mF_2= \dots = \mI$, where the process has a simpler expression $ \mY_n = \sqrt{\mE_{1}} \dots \sqrt{\mE_{n}} \sqrt{\mE_n} \dots \sqrt{\mE_{1}}$.
\begin{proof}
    Note that
    \begin{align}
         & \Exp(\mY_n |\cF_{n-1}) = \sqrt{\mF_{1}} \sqrt{\mE_{1}} \sqrt{\mF_{2}} \sqrt{\mE_{2}}\dots  \sqrt{\mF_{n}} \Exp({\mE_{n}}|\cF_{n-1})  \sqrt{\mF_{n}} \dots \sqrt{\mE_{2}}  \sqrt{\mF_{2}} \sqrt{\mE_{1}}  \sqrt{\mF_{1}} 
        \\
        \matle & \sqrt{\mF_{1}} \sqrt{\mE_{1}} \sqrt{\mF_{2}} \sqrt{\mE_{2}}\dots  \sqrt{\mF_{n}} \mF_n^{-1}  \sqrt{\mF_{n}} \dots \sqrt{\mE_{2}}  \sqrt{\mF_{2}} \sqrt{\mE_{1}}  \sqrt{\mF_{1}}  = \mY_{n-1},
    \end{align}
    concluding the proof.
\end{proof}

{The expression of $\mY_n$ involves splitting the stepwise matrices into square roots and multiplying bidirectionally from the outside, which, we note, ensures the product is still in $\cS_d$ and has the conditional expectation property we want; in particular, $\mF_n \mE_n$ and $\mE_n \mF_n$ may not be symmetric.} The process $\{\mY_n \}$ can be computed incrementally on the fly by noting that
\begin{equation}
    \mY_n = \mL_n \mR_n, \quad \mL_n = \mL_{n-1} \sqrt{\mF_{n}} \sqrt{\mE_{n}}, \quad \mR_n =  \sqrt{\mE_n}  \sqrt{\mF_{n}} \mR_{n-1}.
\end{equation}
We remark that by defining $\mE_n' = \sqrt{\mF_n} \mE_n  \sqrt{\mF_n}$ and $\mF_n' = \mI$, we can obtain another PSD supermartingale $\sqrt{\mE_{1}'} \dots \sqrt{\mE_{n}'} \sqrt{\mE_n'} \dots \sqrt{\mE_{1}'}$. This option, however, 
involves slightly more matrix product computation than \eqref{eqn:e-to-nsm}.

\begin{algorithm}[!h]
\caption{Sequential Test via a Matrix Test Supermartingale}\label{alg:seqtest}
\begin{algorithmic}
\Require{Null hypothesis $\cH_0$ on the data stream $X_1,X_2,\dots$ taking values in $\cX$; \\  Functions $\mE: \cX \to \cS_d^+$ and $\mF: \cX \to \cS_d^{++}$ such that 
$\Exp(\mE(X_n) | X_1,\dots, X_{n-1}) \matle \mF^{-1}(X_n)$ whenever $\cH_0$ is true;
\\
Type-I error level $\alpha \in (0, 1)$; $\mA \in \cS_d^{++}$ such that $\tr (\mA^{-1}) = \alpha$}
\Ensure{$\cH_0$ is falsely rejected with probability at most $\alpha$}

\vskip .5em
\State{$\mL \gets \mI$, $\mR \gets \mI$}
\For{$n = 1,2,\dots$}
\State \textbf{observe} $X_n$
\State $\mL \gets \mL \sqrt{\mF(X_{n})} \sqrt{\mE(X_{n})}$
\State $\mR \gets  \sqrt{\mE(X_n)}  \sqrt{\mF(X_{n})} \mR$
\If{  $\mL \mR \nmatle \mA$ }
\State \textbf{reject} $\cH_0$
\EndIf
\EndFor

\end{algorithmic}
\end{algorithm}

{Recall from \cref{sec:scalars} the concept of testing a null hypothesis $\cH_0$ with a scalar test supermartingale $\{Y_n\}$: a process with $Y_0 = 1$ that is a nonnegative supermartingale under $\cH_0$.} With Lemma~\ref{lem:mateval}, we can recover matrix versions of several scalar test supermartingales 
in the recent sequential statistics literature. In the next subsection, we will state some conditions on a sequence of $\cS_d$-valued random matrices $\{\mX_n\}$ that will lead to two sequences of $\{ \mE_n \}$ and $\{ \mF_n \}$ that satisfy Lemma~\ref{lem:mateval}, so that a $\cS_d^+$-valued supermartingale $\{ \mY_n \}$ can be constructed accordingly via \eqref{eqn:e-to-nsm}.
 Combining with \eqref{eqn:mvi-new}, we form the following generic scheme for sequentially testing a null hypothesis $\cH_0$ using a PSD supermartingale outlined in \cref{alg:seqtest}, where if $\cH_0$ is true the procedure has at most $\alpha$ probability ever rejecting it.

We further remark that, while with scalar nonnegative supermartingales formed multiplicatively, one often works on the log scale to ease the computation. That is, if $M_n = \prod_{i=1}^n E_i$ is a scalar nonnegative supermartingale, one often adds up $\log E_1 + \dots + \log E_n$ and rejects the null when this sum exceeds $\log(1/\alpha)$. With the PSD supermartingale formed in Lemma~\ref{lem:mateval}, while it is indeed true that
\begin{equation}
    \Pr( \exists n, \; \log \mY_n \nmatle \log \mA ) \le  \Pr( \exists n, \;  \mY_n \nmatle  \mA ) \stackrel{ \text{\eqref{eqn:mvi-new}} }{\le} \tr(\mA^{-1})
\end{equation}
due to the monotonicity of the matrix logarithm (Fact~\ref{fct:opmonofuns}), one can not compute $\log \mY_n$ via summing $\log \mE_i$ and $\log \mF_i$. An additive alternative to \eqref{eqn:mvi-new} that avoids matrix multiplication might be of interest for future work.

\subsection{Examples: Conditional Mean Testing Problems}\label{sec:examples-mean}

We have numerous examples for Lemma~\ref{lem:mateval} where one can conduct sequential hypothesis testing with \cref{alg:seqtest}. All examples below, we note, involve the null hypothesis in the form of ``the conditional mean of $\mX_n$ equals $\mM_{n}$'' where $\{\mM_n\}$ is a pre-specified predictable sequence. A common special case is when the matrices $\{ \mX_n \}$ are i.i.d.\ and one tests if their common mean equals a specific value. {It is not straightforward to see that some of these examples satisfy the assumptions of Lemma~\ref{lem:mateval}, so we prove them in \cref{sec:pf-tp}.}

First,
we have the moment generating function (MGF)-based supermartingales whose well-known scalar case is mentioned by e.g.\ \citet[Fact 1(a)]{howard2020time}.

\begin{example}[Matrix MGF Supermartingale]
    Let $\{\mX_n\}$ be an adapted sequence of $\cS_d$-valued random matrices with conditional means $\Exp (\mX_n | \cF_{n-1}) = \mM_n$ and matrix-valued conditional MGF upper bound $\mG_n(\gamma) \matge \Exp (\e^{\gamma(\mX_n - \mM_n)} | \cF_{n-1})$. Then, for each $n$, we can take $\mE_n = \e^{\gamma_n(\mX_n - \mM_n)}$, $\mF_n = \mG_n(\gamma_n)^{-1}$ where $\{ \gamma_n \}$ are predictable scalars, in Lemma~\ref{lem:mateval}.
\end{example}

It is worth noting that in general it is hard to prove matrix MGF bounds and they are known only in very few cases. On the other hand, much more is known on bounding the matrix \emph{log}-MGF $\log \left(\mG_n(\cdot) \right)$ under various distribution assumptions, obtained via the Lieb-Tropp technique mentioned in \cref{sec:spectral-bounds}. These log-MGF upper bounds only lead to scalar-valued supermartingales \eqref{eqn:traceproc} instead of matrix-valued supermartingales. We review some existing results of matrix MGFs by \cite{tropp2012user,tropp2015introduction} in \cref{sec:mgf-logmgf}.


Our second example extends the wealth process of \emph{betting} against the fairness of a coin. The scalar case was initially proposed for Bernoulli coin tosses, generalized to any bounded random scalars by \citet[Chapters 3.1 and 3.2 resp.]{shafer2005probability}, and recently applied to sequential mean testing by \cite{waudby2020estimating}. We have the following corresponding result for bounded random matrices.

\begin{example}[Matrix Betting] \label{ex:bet}
    Let $\{\mX_n\}$ be an adapted sequence of 
 $\cS_d^{+}$-valued random matrices with conditional means $\Exp (\mX_n | \cF_{n-1}) = \mM_n$ and upper bounds $\mB_n$ that are predictable (i.e., $\mX_n \matle \mB_n$ a.s., and $\mB_n$ is $\cF_{n-1}$-measurable).  We can take, in Lemma~\ref{lem:mateval},
    \begin{equation}
        \mE_n =  \mI + \gamma_n (\mX_n - \mM_n), \quad \mF_n = \mI,
    \end{equation}
    where 
    \begin{equation}\label{eqn:bet-fraction}
      -\frac{1}{\lambda_{\max}(\mB_n-\mM_n)} <  \gamma_n < \frac{1}{\lambda_{\max}(\mM_n)}
    \end{equation}
    is predictable.
\end{example}
To see that this is a valid $(\mE_n,\mF_n)$ pair for Lemma~\ref{lem:mateval},
it is clear that  $\Exp (\mE_n | \cF_{n-1}) =  \mI$ and  \eqref{eqn:bet-fraction} ensures $\mE_n \matge 0$. {The resulting $\{\mY_n\}$ is therefore an $\cS_d^+$-valued matrix martingale.
\begin{remark}\label{rmk:bet}
Using Example~\ref{ex:bet} to test the null $\cH_0: \Exp(\mX_n|\cF_{n-1}) =\mM_n $ with Algorithm~\ref{alg:seqtest}, while any choice of predictable $\{\gamma_n\}$ that satisfies \eqref{eqn:bet-fraction} makes sure the type-I error $\le \alpha$ under $\cH_0$, different choices of the sequence lead to difference in \emph{power} when $\cH_0$ does not hold. A simple example is stated below.

Suppose $\Exp (\mX_n|\cF_{n-1}) = \mM_n + c \mI$ for some unknown $c$ ($c=0$ under null $\cH_0$). Since $\Exp[\mE_n|\cF_{n-1}] = (1 + \gamma_n c)\mI$, we can see from the proof of Lemma~\ref{lem:mateval} that the ``growth'' event $\Exp (\mY_n|\cF_{n-1}) \matge \mY_{n-1}$ happens if and only if the predictable $\gamma_n$ and the ground-truth $c$ have matching signs: $\gamma_n c \ge 0$. If the sign of $c$ is not known beforehand, there are two ways to achieve this:
\begin{itemize}
    \item Since $\gamma_n$ can be picked predictably, one can estimate $c$ from $\mX_1,\dots, \mX_{n-1}$ and match the sign of $\gamma_n$ with this estimator. For example, $\hat c_{n-1} = \tr((\mX_1 - \mM_1) + \cdots + (\mX_{n-1}-\mM_{n-1}))/(d(n-1))$, and $\gamma_n = \sg(\hat c_{n-1}) \gamma $ for some $0 < \gamma < \min\left\{ \frac{1}{\lambda_{\max}(\mB_n-\mM_n)}, \frac{1}{\lambda_{\max}(\mM_n)}  \right\} $.
    \item (Hedging). On the other hand, one may launch \emph{two} test martingales, $\{ \mY_n^+ \}$ with $\gamma_n > 0$, and $\{ \mY_n^- \}$ with $\gamma_n < 0$ for all $n$. Since one of them must match the correct sign of $c$ if $c \neq 0$, one can use $\{(\mY_n^+ + \mY_n^-)/2\}$ as the test process (which is still a test martingale for the null $\cH_0:c=0$) to have power in detecting any $c\neq 0$.
\end{itemize}
The same logic applies to many of the subsequent examples.
\end{remark}
}

Besides these two examples where we extend prior scalar results to matrices, we also note that
\citet[Section 6.5]{howard2020time} prove a total of eight inequalities for $\cS_d$-valued martingale differences $\{ \mZ_n \}$ in the style of
\begin{equation}
    \Exp (\exp( \mZ_n -  \mC_n) | \cF_{n-1}) \matle \exp(  \mC_n'  ),
\end{equation}
for some respectively $\cF_{n}$-, $\cF_{n-1}$-measurable matrices $\mC_n, \mC_n'$ under eight different conditions. 
These directly lead to $\cS_d^+$-valued supermartingales via Lemma~\ref{lem:mateval}. However, \cite{howard2020time} do \emph{not} consider combining these conditional expectation bounds into \emph{matrix}-valued supermartingales via \eqref{eqn:e-to-nsm}; they pursue scalar supermartingales instead, with methods we discuss later in \cref{sec:sclr}.
Of their eight inequalities,
we will introduce two interesting cases now and (re-)prove them in \cref{sec:pf-tp}, omitting the other six due to their similar construction.

First, the self-normalized supermartingale that only assumes the second moment exists \citep[Part (f) of Section 6.5]{howard2020time}.

\begin{example}[Finite-Variance Self-Normalized Matrix Supermartingale]\label{ex:sn}
     Let $\{\mX_n\}$ be an adapted sequence of $\cS_d$-valued random matrices with conditional means $\Exp (\mX_n | \cF_{n-1}) = \mM_n$ and conditional variance upper bounds $\Exp( (\mX_n - \mM_n)^2 | \cF_{n-1} ) \matle \mV_n $ that are predictable. We can take, in Lemma~\ref{lem:mateval},
\begin{equation}\label{eqn:sn-ef}
    \mE_n =  \exp(\gamma_n (\mX_n - \mM_n) - \gamma_n^2 (\mX_n - \mM_n)^2/6), \quad \mF_n = \exp \left( -\frac{\gamma_n^2}{3}\mV_n \right)
\end{equation}
where $\{ \gamma_n \}$ are predictable positive scalars.
\end{example}
 {As a bibliographical remark, we follow the usage of \cite{de2009self,Bercu2019SN,Bercu2019SN2}, among others, in adopting the name ``self-normalized''. These authors refer to a martingale $\{M_n\}$ on $\{ \cF_n \}$ as ``self-normalized'' if maximal inequalities of $\{Y_n\}$ are derived via the predictable quadratic variation $\sum \Exp( (Y_n - Y_{n-1})^2  |\cF_{n-1})$  and  the total quadratic variaiton $\sum (Y_n - Y_{n-1})^2 $. The summands of these two quadratic variations correspond loosely to the $\mV_n$ and $(\mX_n - \mM_n)^2$ terms in \eqref{eqn:sn-ef}. The next example further demonstrates this concept.
 }

Second, a bound for symmetric random matrices that have a symmetric (around $0\in \cS_d$), potentially heavy-tailed distribution \citep[Part (d) of Section 6.5]{howard2020time}.

\begin{example}[Symmetric Self-Normalized Matrix Supermartingale]\label{ex:sym}
    Let $\{\mX_n\}$ be an adapted sequence of $\cS_d$-valued random matrices with conditional means $\Exp (\mX_n | \cF_{n-1}) = \mM_n$ and satisfy $\mX_n - \mM_n | \cF_{n-1} \stackrel{d}= \mM_n - \mX_n | \cF_{n-1}$. Then, we can take, in Lemma~\ref{lem:mateval},
\begin{equation}
    \mE_n = \exp \left( \gamma_n (\mX_n-\mM_n) - \frac{\gamma_n^2 (\mX_n-\mM_n)^2}{2} \right), \quad \mF_n = \mI,
\end{equation}
where $\{ \gamma_n \}$ are predictable positive scalars.
\end{example}

Further, let us present two examples in the light- and heavy-tail regimes respectively that do not appear among the ones due to \citet[Section 6.5]{howard2020time}. 
First, a matrix extension of empirical Bernstein inequalities, in the style of \citet[Theorem 4]{howard2021time}, following whom we define the function 
\begin{equation}
    \psiE(x) = -\log(1-x) - x 
\end{equation}
for $x \in [0,1)$. The notation comes from the fact that $\psiE$ is the log-MGF of a centered standard exponential random variable. The following is a rephrasing of a recent result by \citet[Lemma 4.1]{wang2024sharp} in the language of our Lemma~\ref{lem:mateval}.
\begin{example}[Empirical Bernstein Matrix Supermartingale]\label{ex:emp-bern-ms}
    Let $\{\mX_n\}$ be $\cS_d$-valued random matrices adapted to $\{ 
\cF_n \}$ with conditional means $\Exp (\mX_n | \cF_{n-1}) = \mM_n$. Further, suppose there is a predictable and integrable sequence of $\cS_d$-valued random matrices $\{ \widehat \mX_n \}$ such that $\lambda_{\min}(\mX_n - \widehat \mX_n ) \ge -1$. Then, we can take, in Lemma~\ref{lem:mateval},
\begin{equation}
     \mE_n =  \exp(\gamma_n(\mX_n - \widehat{\mX}_n)  - \psiE(\gamma_n) (\mX_n - \widehat{\mX}_n)^2) ,\quad \mF_n = \exp(-\gamma_n (\mM_n - \widehat{\mX}_n)  ),
\end{equation}
where $\{ \gamma_n \}$ are predictable $(0,1)$-valued scalars.
\end{example}

Let us briefly explain the name ``empirical Bernstein''. The original scalar version of the inequality above (that $\Exp(\mE_n | \cF_{n-1}) \matle \mF_n^{-1}$) was proved by \citet[Theorem 4]{howard2021time}, generalized and studied closer by \citet[Theorem 2]{waudby2020estimating}, producing a confidence sequence for the common conditional mean of martingale dependent, $[0,1]$-valued scalars that is asymptotically as sharp as the Bernstein inequality without assuming knowing the underlying variance while adapting to it, and beats numerous then-existing empirical Bernstein inequalities in the literature. The matrix inequality that $\Exp(\mE_n | \cF_{n-1}) \matle \mF_n^{-1}$ above was recently proved by \citet[Lemma 4.1]{wang2024sharp}, who applied the same scalarization technique as \cite{howard2020time} (see \cref{sec:sclr}) to obtain a scalar test supermartingale and a closed-form spectral empirical Bernstein inequality for $\cS_d^{[0,1]}$-valued random matrices, of similar sharpness.

Here, on the other hand, we demonstrate that if a closed-form spectral confidence set is not of interest, one can sequentially test the null $\Exp[\mX_n|\cF_{n-1}] = \mM$, under the assumption that $\{\mX_n\}$ all take values in $\cS_d^{[0,1]}$, by constructing a matrix-valued test supermartingale using Lemma~\ref{lem:mateval} of potential larger power compared to the scalarization approach of \cite{wang2024sharp}. The hyperparameters (predictable scalars $\{ \gamma_n \}$ and matrices $\{\widehat{\mX}_n\}$) in Example~\ref{ex:emp-bern-ms} can be set according to the discussion of \citet[Section 4]{wang2024sharp}.


Our last example generalizes the scalar ``Catoni supermartingales'' due to \cite{catoni2012challenging} in the fixed-time case and \cite{wang2023catoni} in the sequential case. With scalars, this influence function-based technique produces subGaussian confidence intervals and sequences for heavy-tailed random variables with a $p$\textsuperscript{th} moment upper bound where $p \in (1,2]$. In the $p=2$ special case \cite{wang2023catoni} demonstrate that the ``Catoni-style confidence sequence'' is tighter than the self-normalization technique which we already generalized to matrices in Example~\ref{ex:sn},
under a variance upper bound assumption. Our matrix extension of these results starts with the standard definition of Catoni's influence function (see e.g.\ \citet[Section 1]{chen2021generalized}, \citet[Section 9]{wang2023catoni}).

\begin{definition}
    An increasing function $\phi:\mathbb R \to \mathbb R$ is \emph{upper $p$-logarithmically contractive} if $\phi(x) \le \log(1+x+|x|^p/p)$; \emph{lower $p$-logarithmically contractive} if $-\log(1-x+|x|^p/p) \le \phi(x)$.
\end{definition}

For example, the function
\begin{equation}
    \phi_p (x) = \begin{cases} \log(1 + x + x^p/p), & x \ge 0; \\ -\log(1 - x + (-x)^p/p), & x < 0. \end{cases}
\end{equation}
is both upper and lower $p$-logarithmically contraction.  

\begin{example}[Matrix Catoni Supermartingales]\label{ex:cat}
     Let $\{\mX_n\}$ be an adapted sequence of $\cS_d$-valued random matrices with conditional means $\Exp (\mX_n | \cF_{n-1}) = \mM_n$ and conditional $p$\textsuperscript{th} moment upper bounds $\Exp( (\abs(\mX_n - \mM_n))^p | \cF_{n-1} ) \matle \mV_{n} $ that are predictable. For any upper $p$-logarithmically contractive function, we can take in Lemma~\ref{lem:mateval},
\begin{equation}
    \mE_n = \exp(\phi (  \gamma_n(\mX_n - \mM_n) )  ) , \quad \mF_n =  \left ( \mI + \frac{\gamma_n^{p}}{p} \mV_n \right )^{-1};
\end{equation}
and for any lower $p$-logarithmically contractive function, we can take in Lemma~\ref{lem:mateval},
\begin{equation}
    \mE_n = \exp(- \phi (  \gamma_n(\mX_n - \mM_n) )  ) , \quad \mF_n =  \left ( \mI + \frac{\gamma_n^{p}}{p} \mV_n \right )^{-1}.
\end{equation}
Above, $\{ \gamma_n \}$ are predictable positive scalars.
\end{example}

Example~\ref{ex:cat} is proved in \cref{sec:pf-tp}. We remark that, in the scalar case, the two supermartingales constructed this way remain supermartingales under \emph{one-sided nulls}, one for each side. This fails to hold with matrices (e.g., the null that $\Exp (\mX_n | \cF_{n-1}) \matle \mM_n$) as there is no guarantee that $\exp(\pm\phi(\cdot))$ are operator monotone. Still, the two matrix test supermartingales play the role of detecting null-deviating behavior of opposite directions, and
a powerful test for the point null $\Exp (\mX_n | \cF_{n-1}) = \mM_n$ can be formed by a convex combination of the two.


\subsection{Matrix e-Values and e-Processes}\label{sec:mat-e-proc}

We have thus far generalized scalar nonnegative supermartingales to PSD matrix supermartingales, as well as their numerous applications in sequential statistics as ``test supermartingales''. However, with scalars, it is known that nonnegative supermartingales fail to be powerful in some examples involving large composite null hypotheses, and a wider class of processes called \emph{e-processes} is often used instead \citep{ramdas2022testing}. Loosely speaking, an e-process is a nonnegative process upper bounded by supermartingales. We first recall the formal definitions of scalar e-values and e-processes, and generalize them to matrices as well.

Throughout this subsection, we consider $X_1,X_2,\dots$ a sequence of random observations taking values in some space $\cX$ and they generate the filtration $\{\cF_n\}$.
We use notations $\Pr_{P}$, $\Exp_{P}$ and the general terminology ``for $P$" (e.g.\ we may say an adapted stochastic process is a martingale for $P$) to refer to probabilistic statements when the observations $X_1,X_2,\dots$ are distributed jointly according to $P$, where $P$ is some distribution on $\cX \times \cX \times \dots$.

Let $\cP$ be a set of distributions on $\cX \times \cX \times \dots$. {An $\cF_{\infty}$-measurable nonnegative random variable $Y$ is called an e-value for $\cP$ if, for any $P \in \cP$, its mean $\Exp_{P}Y \le 1$.}
An adapted nonnegative process $\{Y_n\}$ is called an e-process for $\cP$ if, {for any stopping time $\tau$, the stopped variable $Y_\tau$ is an e-value for $\cP$. It is known that, due to \citet[Lemma 6]{ramdas2020admissible}, $\{Y_n\}$ is an e-process 
if and only if for any $P \in \cP$ it is upper bounded by a supermartingale $\{M_n^P\}$ with $M_0^P \le 1$.}
Clearly, if $\{Y_n\}$ is a nonnegative supermartingale for $\cP$ (i.e.\ for every $P \in \cP$), 
the optional stopping theorem implies that
it is also an e-process for $\cP$, but e-processes are a much more general class~\citep{ramdas2022testing,ruf2023composite}.

Scalar e-processes also satisfy \eqref{eqn:vi}, both time-uniformly and at a stop time:
\begin{gather}
    \text{For any $P \in \cP$}: \quad \Pr_P( \exists n, \ Y_n \ge \sdiv ) \le  \sdiv ^{-1};\\
     \text{and for any stopping time $\tau$},  \quad \Pr_P( Y_\tau  \ge \sdiv ) \le \sdiv^{-1}.   
\end{gather}
This property enables scalar e-processes to also be used for testing the null hypothesis that the data is drawn from some $P \in \cP$: we may reject the null as soon as the e-process exceeds $1/\alpha$ (for some prespecified constant $\alpha \in (0,1)$) and be guaranteed that the type-I error is at most $\alpha$. E-processes are particularly useful for composite sequential testing problems for which nonnegative supermartingales are powerless 
 \citep{ramdas2022testing,ruf2023composite}.



We define matrix e-values and e-processes similarly.
\begin{definition}[Matrix e-Value]\label{def:mat-ev}
    Let $\cP$ be a set of distributions on $\cX \times \cX \times \dots$. An $\cS_d^+$-valued, $\cF_{\infty}$-measurable random matrix $\mY$ is called a (matrix) e-value for $\cP$ if $\Exp_P \mY \matle \mI$ for any $P \in \cP$. 
\end{definition}

\begin{definition}[Matrix e-Process]\label{def:mat-epr}
Let $\cP$ be a set of distributions on $\cX \times \cX \times \dots$. An $\cS_d^{+}$-valued adapted stochastic process $\{\mY_n\}$ is called a \emph{(matrix) e-process} for $\cP$ if, for any stopping time $\tau$, the stopped random matrix $\mY_\tau$ is a matrix e-value for $\cP$.
\end{definition}

It can be seen from the matrix optional stopping theorem (Theorem~\ref{thm:matrix-optional-stopping}) that if $\{ \mY_n \}$ is a PSD supermartingale for $\cP$ then it is also an e-process for $\cP$. More generally, the following is a sufficient condition for a matrix process to be an e-process:

\begin{proposition}\label{prop:epr-upper-bounded}
    Suppose for any $P \in \cP$, there exists an $\cS_d^+$-valued supermartingale $\{\mM_n^{P}\}$ for $P$  with $\mM_0^{P} \matle \mI$ that upper bounds the process $\{\mY_n\}$, i.e., $\mM_n^{P} \matge \mY_n$ for all $n$ almost surely for $P$.  Then, $\{\mY_n\}$ is a matrix e-process for $\cP$.
\end{proposition}
This can be seen by simply noting that $\Exp_P \mY_\tau \matle  \Exp_P \mM_\tau^P \matle \mI$ for any $P,\tau$. While the converse to Proposition~\ref{prop:epr-upper-bounded} (that it is a necessary condition for e-processes as well) is true in the scalar case \citep[Lemma 6]{ramdas2020admissible}, the argument fails to extend to matrices as the construction for scalars relies on the Snell envelope which we have yet to find a suitable generalization for matrices due to the non-existence of maximum and minimum under $\matle$ within $\cS_d$.

Similar to the scalar case above, matrix e-processes satisfy \eqref{eqn:mvi-new}, both time-uniform and at a stopping time:
\begin{gather}
    \text{For any $P \in \cP$}: \quad \Pr_P( \exists n, \ \mY_n \nmatle \mdiv ) \le \tr(\mdiv^{-1});\\
     \text{and for any stopping time $\tau$},  \quad \Pr_P( \mY_\tau  \nmatle \mdiv ) \le \tr(\mdiv^{-1}),
\end{gather}
{which follows directly from the proof of Theorem~\ref{thm:matvil1}.} Therefore, one can run a procedure similar to \cref{alg:seqtest} with an e-process instead {for sequential hypothesis testing}.

Simple examples of matrix e-processes following the previous subsection {therefore} include lower bounds of $\cS_d^+$-valued supermartingales. {In the scalar case, an e-process is often constructed for a composite null hypothesis by taking an infimum over the test supermartingales under each of the null distributions. While a similar construction carries over to the matrices, it is important to note that in general \emph{the} matrix minimum does not exist: for most pairs $\mA, \mB \in \cS_d$,
there is no largest lower bound of $\mA$ and $\mB$ in the $\matle$ order within $\cS_d$ \citep[Theorem 6]{kadison1951order}. However, we can define \emph{a} matrix minimum, for example as follows:
\begin{definition}\label{def:min}
    We define the following minimum operation $\curlyvee : \cS_d \times \cS_d \to \cS_d$:
    \begin{equation}
        \mA \curlyvee \mB = \begin{cases}
            \mA, & \text{if }\mA \matle \mB; \\
            \mA - \lambda_{\max}(\mA -\mB) \cdot \mI, & \text{otherwise}.
        \end{cases}
    \end{equation}
    Note that $\mA \curlyvee \mB \matle \mA$ and $\mA \curlyvee \mB \matle \mB$, and that in general $\mA \curlyvee \mB \neq \mB \curlyvee \mA$.
\end{definition}

A few comments follow this definition of minimum. First, while the operation $\curlyvee$ is not commutative,
one may symmetrize it by defining $\mA \curlyvee^* \mB = \frac{1}{2}(\mA \curlyvee \mB + \mB \curlyvee \mA)$, which, while commutative and is a minimum, is not associative. Second, the non-existence of the ``largest lower bound'' in the $\matle$ order states that there can exist a matrix $\mC$ that satisfies
\begin{equation}
    \mC \matle \mA, \mC \matle \mB, \mC \nmatle \mA \curlyvee \mB.
\end{equation}
{For example, $\mB \curlyvee \mA$ is often an example of $\mC$ that satisfies above. On the other hand, however,} it is easy to verify that there does \emph{not} exist a matrix $\mC$ that satisfies
\begin{equation}
    \mC \matle \mA, \mC \matle \mB, \mC \matgt \mA \curlyvee \mB,
\end{equation}
in the sense that there are several choices of ``matrix minima'', Definition~\ref{def:min} being one among them, and these choices do not dominate one another. Finally, we note that the matrix minimum
 enables us to construct matrix e-processes for null hypotheses that are finite unions of several conditions under which a matrix test supermartingale is constructible:
 \begin{proposition}
     Suppose $\cP = \bigcup_{k=1}^K \cP_k$ and $\{\mY^k_{n}\}$ is a matrix e-process for $\cP_k$. Then, $\{ \mY_n^1 \curlyvee \mY_n^2 \curlyvee \dots \curlyvee \mY_n^K \}$ is a matrix e-process for $\cP$.
 \end{proposition}
 {Here, the $K$-fold minimum  $\mY_n^1 \curlyvee \mY_n^2 \curlyvee \dots \curlyvee \mY_n^K$  can be defined either left-associatively as $(\dots(\mY_n^1 \curlyvee \mY_n^2) \curlyvee \dots) \curlyvee \mY_n^K$ or right-associatively as $\mY_n^1 \curlyvee (\dots \curlyvee (\mY_n^{K-1} \curlyvee \mY_n^K ) \dots)$. These lead to two different e-processes.}
 One can take the average over all $K!$ permutations if one desires an e-process that is symmetric w.r.t.\ all $K$ subclasses.
 A simple example is as follows.}
Let $ \mM,\mM' \matle \mB$ be three matrices in $\cS_d^{++}$, and we define $\cB(\mB, \mM)$ to be the set of distributions on $\cS_d^{+}$ upper bounded by $\mB$ with mean $\mM$. Further, we define
the distribution families
\begin{equation}
    \cP_1 = \{  P_1\otimes P_2\otimes \dots : P_n \in \cB(\mB, \mM) \text{ for all }n     \}
\end{equation}
and
\begin{equation}
    \cP_2 = \{  P_1\otimes P_2\otimes \dots : P_n \in \cB(\mB, \mM') \text{ for all }n     \}.
\end{equation}
According to Example~\ref{ex:bet}, we can construct a $\cS_d^{+}$-valued betting process $\{ \mY_n \}$ which is a supermartingale under any of $P \in \cP_1$, and another $\{ \mY_n' \}$ which is a supermartingale under any of $P \in  \cP_2$. Then the minimum process $\{  \mY_n \curlyvee \mY_n'  \}$ is a matrix e-process for $\cP_1 \cup \cP_2$, useful for testing sequentially the null that the true mean is always \emph{either} $\mM$ \emph{or}
 $\mM'$.



\section{Backward PSD Submartingales}\label{sec:bw}

We have so far discussed the properties and applications $\cS_d^+$-valued supermartingales on a forward filtration $\{\cF_n\}_{n \ge 0}$. We now move on to a backward filtration $\{ \cG_0 \}_{n \ge 0}$ and its adapted $\cS_d^+$-valued submartingales to extend the established theory with scalars. This involves first considering a forward submartingale of finite length on the finite forward filtration $\{\cF_n\}_{0 \le n \le N}$, inverting the time, and extending indefinitely into the ``past'' (before the inversion). For this reason, we shall always be explicit with the filtration and the time range in this section.

\subsection{Submartingales and Maximal Inequalities: From Forward to Backward}

In the proof of Theorem~\ref{thm:matrix-optional-stopping},
we appeal to the simple fact that if $\{ \mY_n \}$ is a matrix supermartingale then $\{ \vv\trsp \mY_n \vv \}$ is a scalar supermartingale, which enables us to derive stopping time theorems for matrix supermartingales via those for scalar supermartingales. This clearly
remains still valid if we replace ``supermartingale"s with ``submartingale"s. To wit, we say that $\{ 
\mY_n \}_{n\ge 0}$ adapted to $\{ \cF_n \}_{n \ge 0}$ is a matrix submartingale if $\Exp (\mY_n | \cF_{n-1} ) \matge \mY_{n-1}$. It is clear from the proof of Theorem~\ref{thm:matrix-optional-stopping} that,
\begin{lemma}\label{lem:matrix-optional-stoppong-sub}
    Let $\{ \mY_n \}_{n\ge 0}$ be an $\cS_d$-valued submartingale, and $\tau$ be a stopping time on $\{ \cF_n \}_{n\ge 0}$. Then, $\{ \mY_{n \wedge \tau} \}_{n\ge 0}$ is also a submartingale on $\{ \cF_n \}_{n\ge 0}$. Consequently, if $\tau \le N$ almost surely for some $N \in \mathbb N$, then $\Exp \mY_\tau  \matle \Exp \mY_N$. 
\end{lemma}
Doob's submartingale inequality, which is ``dual" to Ville's supermartingale inequality, is also matrix generalizable.
\begin{theorem}[Matrix Doob's inequality]\label{thm:matdoob}
    Let $\{ \mY_n \}$ be an $\cS_d^+$-valued submartingale. For every $N \in \mathbb N$ and every $\mA \in \cS_d^{++}$,
    \begin{equation}\label{eqn:matdoob}
        \Pr\left( \exists n \le N, \mY_n \nmatle \mA \right)  \le \tr ( \Exp (\mY_N\id_{ \{ \exists n \le N, \mY_n \nmatle \mA\} }) \mA^{-1}  ) \le \tr( (\Exp \mY_N) \mA^{-1}  ).
    \end{equation}
\end{theorem}
\begin{proof}
    Let $E$ be the event $\{ \exists n \le N, \mY_n \nmatle \mA  \}$. Define the stopping time 
    \begin{equation}
    \tau  = \inf \{n: \mY_n \nmatle \mA  \} \wedge N ,    
    \end{equation}
    so it satisfies $\tau \le N$ and
    \begin{equation}
        \{ \mY_\tau \id_{ E } \nmatle \mA  \} = E,\quad \mY_\tau \id_{\overline{E}}= \mY_N \id_{\overline{E}}.
    \end{equation}
    Thus by Lemma~\ref{lem:matrix-optional-stoppong-sub} and \eqref{eqn:mmi},
    \begin{align}
     &   \tr( \mA^{-1/2} (\Exp \mY_N
) \mA^{-1/2} ) \ge \tr( \mA^{-1/2} (\Exp \mY_\tau) \mA^{-1/2} ) 
\\
= & \tr( \mA^{-1/2}\Exp (\mY_\tau \id_E) \mA^{-1/2} ) +  \tr( \mA^{-1/2}\Exp (\mY_\tau \id_{\overline{E}})\mA^{-1/2})
\\
\ge & \Pr( \mY_\tau \id_{ E } \nmatle \mA ) +   \tr( \mA^{-1/2}\Exp (\mY_N \id_{\overline{E}})\mA^{-1/2}),
    \end{align}
   implying that $\Pr (E) \le \tr( \mA^{-1/2} \Exp (\mY_N \id_E )
 \mA^{-1/2} )$.
\end{proof}

The backward Ville's inequality comes from the idea that Doob's inequality
remains valid even without assuming that time starts at $n=0$ --- or that it starts \emph{at all}.  In the scalar regime, this idea proves the backward Ville's inequality \eqref{eqn:bvi} and proves that it is equivalent to the (forward) Doob's inequality (hence the ``duality'' between Ville's and Doob's inequality; see the second proof of Theorem 2 in \cite{manole2023martingale}).

This idea is valid for matrices as well; that is, Theorem~\ref{thm:matdoob} 
is applicable even if $\{ \mY_n \}$ and $\{ \cF_n \}$ are indexed on $n \in \{ \dots, -2, -1, 0 ,1, \dots, N \}$. We formalize this by recalling that a backward filtration  $\{ \cG_n \}$ is one that satisfies $ \cG_0 \supseteq \cG_1 \supseteq \dots $ and an adapted \emph{backward} submartingale $\{ \mY_n \}$ is one such that each $\mY_n$ is $\cG_{n}$-measurable and $\Exp[ \mY_n | \cG_{n+1} ] \matge \mY_{n+1}$, while we formally keep the time indices $n \ge 0$.  

\begin{theorem}[Matrix Backward Ville's Inequality]\label{thm:matbwvil}
    Let $\{ \mY_n \}_{n\ge 0}$ be an $\cS_d^+$-valued backward submartingale adapted to the backward filtration $\{ \cG_n \}_{n\ge 0}$. 
    For any $\mA \in \cS_d^{++}$,
    \begin{equation}\label{eqn:mbvi}
        \Pr( \exists n, \; \mY_n \nmatle \mA) \le \tr( (\Exp \mY_0) \mA^{-1} ).
    \end{equation}
\end{theorem}
\begin{proof}
    Take an arbitrary $N \in \mathbb N$ and consider, indexed by $0\le n\le N$, the process $\mP_n = \mY_{N-n}$ on the (forward) filtration $\cF_n = \cG_{N-n}$. It is clear that $\{ \mP_n \}_{0\le n \le N}$ is a (forward) submartingale on $\{ \cF_n \}_{0\le n \le N}$, to which we can apply Theorem~\ref{thm:matdoob},
    \begin{equation}
         \Pr\left( \exists n \le N, \mP_n \nmatle \mA \right)  \le \tr( (\Exp \mP_N) \mA^{-1}  ).
    \end{equation}
    Substituting {$\mY_n$} back to the inequality above, we have,
    \begin{equation}
         \Pr\left( \exists n \le N, \mY_n \nmatle \mA \right)  \le \tr( (\Exp \mY_0) \mA^{-1}  ).
    \end{equation}
    The desired inequality {follows from the monotone convergence theorem:
   the random variable $\id_{\{ \exists n, \; \mY_n \nmatle \mA \}}$ is the pointwise limit of the increasing sequence of random variables $\{ \id_{\{ \exists n \le N, \; \mY_n \nmatle \mA \}} \}_{N \ge 1}$.
    }
\end{proof}

\subsection{Application: Concentration of Exchangeable Matrices}

An immediate application of Theorem~\ref{thm:matbwvil} is numerous matrix Chebyshev-type concentration bounds for \emph{exchangeable} random matrices. We recall that a sequence of random variables $X_1,X_2,\dots$ taking values in a Polish space is exchangeable if the joint distribution of $(X_{i_1},\dots, X_{i_k})$ equals that of $(X_{j_1},\dots, X_{j_k})$ for any distinct $i_1,\dots , i_k$ and distinct $j_1,\dots, j_k$. Repetitions of a single random variable, the P\'olya urn model, sampling from a finite set without replacement, as well as i.i.d.\ random variables are all examples of exchangeable sequences.

The scalar exchangeable Chebyshev inequality was recently proved by \citet[Theorem 2]{ramdas2023randomized}. 
Recall from Fact~\ref{fct:opmonofuns} that when $p \in [1,2]$, the function $\mX \mapsto \mX^p$ is operator convex within the domain $\cS_d^+$ (for $1\le p < 2$) or $\cS_d$ (for $p=2$) respectively,
\begin{equation}\label{eqn:power-op-cov}
   \frac 1 n ( \mX_1^p + \dots +  \mX_n^p ) \matge \left( \frac{\mX_1+\dots + \mX_n}{n} \right)^p.
\end{equation}
This gives rise to the following backward submartingale for exchangeable matrices.

\begin{lemma}\label{lem:pth-moment-bw}
    Let $p\in[1,2]$, and $\{ \mX_n \}$ be an exchangeable sequence of $\cS_d$-valued random matrices {with $\Exp \|\mX_1\|^p < \infty$}. Suppose $\mC \in \cS_d$ is deterministic, and either of the following holds:
    \begin{itemize}
        \item $p = 2$; or
        \item $\mX_n \matge \mC$ almost surely for all $n$.
    \end{itemize}
    Then, the process
\begin{equation}
    \mR_n = ( \overline{\mX}_n - \mC )^p
\end{equation}
 is a backward submartingale adapted to the exchangeable backward filtration
\begin{equation}
 \cE_n = \sigma( f(\mX_1, \mX_2,\dots) : f \text{ symmetric w.r.t.\ first $n$ arguments} ). 
\end{equation}
\end{lemma}

\begin{proof} Recalling that the conditional expectation w.r.t.\ the exchangeable $\sigma$-algebra $\cE_{n+1}$ equals the average over all $(n+1)!$ permutations of the indices $1,2,\dots,n+1$ (see e.g.\ \citet[Theorem 12.10]{klenke2013probability}), we have
    \begin{align}
        & \Exp[ \mR_n | \cE_{n+1} ] = \frac{1}{(n+1)!} \sum_{\sigma \in \mathrm S_{n+1}} \left(\frac{\mX_{\sigma(1)}+\dots + \mX_{\sigma(n)}}{n} - \mC \right)^p
        \\
       \stackrel{\eqref{eqn:power-op-cov}}\matge & \left( \frac{\sum_{\sigma \in \mathrm S_{n+1}} \left(\frac{\mX_{\sigma(1)}+\dots + \mX_{\sigma(n)}}{n} - \mC \right)  }{(n+1)!} \right)^p = \mR_{n+1},
    \end{align}
    concluding the proof that $\{ \mR_n \}$ is a backward submartingale.
\end{proof}

Applying the matrix backward Ville's inequality (Theorem~\ref{thm:matbwvil}) to this process $\{ \mR_n \}$, we immediately have the following two exchangeable matrix Chebyshev inequalities. First, the following holds when a finite variance exists.  Recall that the \emph{anticommutator} between two sqaure matrices $\mA, \mB$ is defined as $\{ \mA, \mB \} = \mA \mB+ \mB \mA$.

\begin{theorem}[Exchangeable Matrix Chebyshev Inequality]\label{thm:xmci}
    Let $\{ \mX_n \}$ be an exchangeable sequence of $\cS_d$-valued random matrices with mean $\mM$ and variance upper bound $\Exp (\mX_1 - \mM)^2 \matle \mV $. Then, for any $\mA \in \cS_d^{++}$,
    \begin{equation}\label{eqn:xmci}
         \Pr( \exists n, \;  \abs(\overline{\mX}_n - \mM)  \nmatle \mA) \le \tr( \mV \mA^{-2} ).
    \end{equation}
    Further, if $\Exp \{\mX_i - \mM, \mX_j - \mM\} = 0 $ for any $i \neq j$, we have for any fixed $N$,
    \begin{equation}\label{eqn:xmci2}
         \Pr( \exists n \ge N, \;  \abs(\overline{\mX}_n - \mM)  \nmatle \mA) \le N^{-1}\tr( \mV \mA^{-2} ).
    \end{equation}
\end{theorem}
\begin{proof}
    Applying Theorem~\ref{thm:matbwvil} to $\mR_n = (\overline{\mX}_n - \mM)^2$, which is a backward submartingale according to Lemma~\ref{lem:pth-moment-bw},
    \begin{equation}
        \Pr( \exists n, \; ( \overline{\mX}_n - \mM )^2 \nmatle \mA^2) \le \tr( (\Exp \mR_1) \mA^{-2} ) \le \tr( \mV \mA^{-2} )  .
    \end{equation}
    Now we use the operator monotonicity of $\mX \mapsto  \mX^{1/2}$ (Fact~\ref{fct:opmonofuns}),
    \begin{equation}
        \Pr( \exists n, \;  \abs(\overline{\mX}_n - \mM)  \nmatle \mA)\le  \Pr( \exists n, \; ( \overline{\mX}_n - \mM )^2 \nmatle \mA^2),
    \end{equation}
    concluding the proof of \eqref{eqn:xmci}.
    To show \eqref{eqn:xmci2}, simply apply Theorem~\ref{thm:matbwvil} to the submartingale $\mR_n = (\overline{\mX_n} - \mM)^2$ with $n \ge N$, obtaining
    \begin{equation}
        \Pr( \exists n \ge N, \; ( \overline{\mX}_n - \mM )^2 \nmatle \mA^2) \le \tr( (\Exp \mR_N) \mA^{-2} ) \le N^{-1} \tr( \mV \mA^{-2} )  
    \end{equation}
    where the last step is due to Lemma~\ref{lem:contr-var}.
\end{proof}

We remark that the assumption $\Exp \{\mX_i - \mM, \mX_j - \mM\} = 0$ for $i\neq j$ is satisfied when $\mX_1, \mX_2, \dots$ are pairwise uncorrelated, which is in turn satisfied under martingale difference assumption (i.e., $\{ \mX_i \}$ is adapted to some filtration $\{ \cF_i \}$ with $\Exp(\mX_i|\cF_{i-1}) = \mM$ and $\Var(\mX_i|\cF_{i-1}) \matle \mV$), which is in turn implied if they are i.i.d.

The inequality \eqref{eqn:xmci} generalizes the matrix Chebyshev inequality by \citet[Theorem 14]{ahlswede2002strong} which states that $\Pr(  \abs(\overline{\mX} - \mM)  \nmatle \mA) \le \tr( \mV \mA^{-2} )$ for \emph{one} square-integrable matrix $\mX$ with variance $\mV$. It is only by assuming  $\Exp \{\mX_i - \mM, \mX_j - \mM\} = 0 $ that we may obtain a ``shrinking'' factor $N^{-1}$ on the tail probability that shows up in the right hand side of \eqref{eqn:xmci2}. Without this assumption, these exchangeable matrices can all be the repetition of a single random matrix $\mX_1 = \mX_2 = \dots$ in which case \eqref{eqn:xmci} is reduced to the matrix Chebyshev inequality by \citet[Theorem 14]{ahlswede2002strong}. On the other hand, with the stated additional assumption $\Exp \{\mX_i - \mM, \mX_j - \mM\} = 0 $,  the inequality \eqref{eqn:xmci2} generalizes the matrix ``weak law of large numbers'' by \citet[Corollary 16]{ahlswede2002strong} which we already quoted as \eqref{eqn:mci}, that $\Pr(  \abs(\overline{\mX}_N - \mM)  \nmatle \mA) \le N^{-1}\tr( \mV \mA^{-2} )$ for $N$ i.i.d.\ matrices $X_1,\dots, X_N$. Both generalizations, we note, offer the exact same concentration at the sample sizes (1 and $N$ respectively) compared to the original results by \cite{ahlswede2002strong}, while also extending infinitely into all future sample sizes. These generalizations generalize those in the scalar case by \citet[Section 2.2]{ramdas2023randomized}. 

We next present results for exchangeable matrices with heavier tails.

\begin{theorem}[Exchangeable Matrix $p$-Chebyshev Inequality]\label{thm:xmpci}  Let $\{ \mX_n \}$ be an exchangeable sequence of $\cS_d^+$-valued random matrices with raw $p$\textsuperscript{th} moment $\mV_p = \Exp \mX_1^p$, where $p\in[1,2]$. Then, for any $\mA \in \cS_d^{++}$,
    \begin{equation}
        \Pr( \exists n, \;  \overline{\mX}_n  \nmatle \mA) \le  \tr( \mV_p \mA^{-p} ).
    \end{equation}
\end{theorem}
\begin{proof}
We define $ \mR_n = ( \overline{\mX}_n  )^p$, which, since $p \in [1,2]$, is a matrix backward submartingale due to Lemma~\ref{lem:pth-moment-bw}.
    Applying Theorem~\ref{thm:matbwvil} to $\{\mR_n\}$, 
    we see that
    \begin{equation}
        \Pr( \exists n, \; (\overline{\mX}_n) ^p \nmatle \mA^p) \le \tr( \mV_p \mA^{-p} ).
    \end{equation}
    Now use the operator monotonicity of $\mX \mapsto  \mX^{1/p}$ (Fact~\ref{fct:opmonofuns}),
    \begin{equation}
        \Pr( \exists n, \;  \overline{\mX}_n  \nmatle \mA) \le  \Pr( \exists n, \; (\overline{\mX}_n) ^p \nmatle \mA^p),
    \end{equation}
    concluding the proof.
\end{proof}

We remark that, from Theorem~\ref{thm:xmci} to Theorem~\ref{thm:xmpci}, we weaken the second moment to a $p$\textsuperscript{th} moment assumption, but in an imperfect manner as we obtain a concentration not around the mean $\mM$ but a Markov-type ``upper tail bound'' for $\cS_d^+$-valued matrices instead. {The raw $p$\textsuperscript{th} moment is used above. Also, $p > 2$ is not allowed. The reason is summarized in \cref{tab:power-convexity}. }
\begin{table}[!h]
    \centering
    \begin{tabular}{c|c|c|c|c}
       $p$  &   $x \mapsto x^p$, $x\ge 0$ &   $\mX \mapsto \mX^p$, $\mX \matge 0$ &   $x \mapsto |x|^p$, $x \in \mathbb R$ &   $\mX \mapsto (\abs\mX)^p$, $\mX \in \cS_d$  \\ \hline
      $[1,2)$    & convex & convex & convex & non-convex \\
      $2$  & convex & convex & convex & convex    \\
      $(2,\infty)$ & convex & non-convex & convex & non-convex
    \end{tabular}
    \caption{Convexity of scalar-to-scalar vs.\ matrix-to-matrix power maps.}
    \label{tab:power-convexity}
\end{table}

Nonetheless, thanks to Fact~\ref{fct:tracejen}, we can resort to trace scalarization to obtain the following spectral and trace inequalities that hold for all $p \ge 1$.

\begin{theorem}[Exchangeable Matrix $p$-Chebyshev Trace Inequality]\label{thm:p-cheb-trace}  Let $\{ \mX_n \}$ be an exchangeable sequence of $\cS_d$-valued random matrices with mean $\mM$ and central $p$\textsuperscript{th} moment $\mV_p = \Exp (\abs(\mX_1-\mM))^p$, where $p \ge 1$. Then, for any $a > 0$,
    \begin{equation}
    \Pr( \exists n, \;  \|\overline{\mX}_n - \mM\| \ge a )  \le     \Pr\left( \exists n, \;  \tr \left( \abs(\overline{\mX}_n - \mM)^p \right) \ge a^p \right) \le  a^{-p} \tr \mV_p .
    \end{equation}
\end{theorem}
The inequality above, which we defer its straightforward proof to \cref{sec:pfbw}, differs from Theorem~\ref{thm:xmpci} in that the concentration is now around the mean $\mM$.

\section{Remarks and Extensions}\label{sec:dis}

\subsection{Tightness of Matrix Markov's and Ville's Inequalities}\label{sec:tight}

Recall that a PSD supermartingale $\{ \mY_n \}$ can be converted into a scalar nonnegative supermartingale via either taking the trace $\{ \tr \mY_n \}$ or a quadratic form $\{ \vv\trsp \mY_n \vv \}$.
We thus discuss the tightness of \eqref{eqn:mvi-new} in this subsection in comparison to the scalar \eqref{eqn:vi} on these transformations. Since \eqref{eqn:mvi-new} is proved by applying \eqref{eqn:mmi} on a stopped random matrix, it suffices to compare $\eqref{eqn:mmi}$ versus \eqref{eqn:mi} on a PSD random matrix $\mX$. Without loss of generality, we assume that $\Exp\mX = \mI$. Denote by $\cC_{\tr}(\alpha)$ the concentration region by applying \eqref{eqn:mi} to $\tr \mX$:
\begin{equation}
    \cC_{\tr}(\alpha) = \{ \mX \in \cS_d^{+} : \tr \mX \le d/\alpha \};
\end{equation}
by $\cC_{\vv}(\alpha)$  the concentration region by applying \eqref{eqn:mi} to $\vv\trsp \mX \vv$ where $\|\vv\| = 1$:
\begin{equation}
    \cC_{\vv}(\alpha) = \{ \mX \in \cS_d^{+} : \vv\trsp \mX \vv \le 1/\alpha \}.
\end{equation}
Let $\cA_\alpha$ be the following set of matrices:
\begin{equation}
    \cA_{\alpha} = \{ \mA \in \cS_d^{++} :  \tr (\mA^{-1}) = \alpha \},
\end{equation}
and we denote by $\cC(\mA)$ the concentration region by applying \eqref{eqn:mmi} to $\mX$ with threshold $\mA \in  \cA_{\alpha} $:
\begin{equation}
    \cC({\mA}) = \{ \mX \in \cS_d^{+} :  \mX \matle \mA \}.
\end{equation}
Following e.g.\ \cite{MINSKER2017111}, we denote by $\er(\mA)$ the effective rank $\tr \mA /\|\mA\|  \in [1,d]$ of a PD matrix.
We can establish the following relations among these concentration regions via elementary linear algebra.
\begin{proposition}\label{prop:tight} Suppose $\alpha \in (0,1)$ and $d \ge 2$.
    \begin{enumerate} 
        \item $\cC_{\tr}(\alpha)  \subsetneq \cC((d/\alpha)\mI)$;
        \item If $\mA \in \cA_{\alpha}$ and $\mA \neq (d/\alpha)\mI$, then $\cC_{\tr}(\alpha)  \nsubseteq \cC(\mA)$ and $\cC(\mA)  \nsubseteq \cC_{\tr}(\alpha)$;
        \item If $\mA \in \cA_{\alpha}$, then $ \cC_{\tr}(\alpha') \subseteq \cC(\mA) $ if and only if $\alpha' \ge  \frac{d}{\er(\mA^{-1})} \cdot \alpha $;
        \item If $\mA \in \cA_{\alpha}$, then $\cC(\mA) \subseteq \cC_{\tr}(\alpha')$ if and only if $\alpha' \le d(\tr \mA)^{-1}$ {$\in(0, \frac{\alpha}{\er(\mA)} ]$};
        \item For any $\vv$ such that $\| \vv \| = 1$ and any $\alpha' > \alpha$, there always exists an $\mA \in \cA_{\alpha'}$ such that $\cC(\mA) \subsetneq \cC_{\vv}(\alpha)$;
        \item For any $\mA \in \cA_{\alpha}$ and any $\alpha' \in (0,1)$, there exists no $\vv$ such that $\| \vv \| = 1$ and $\cC_{\vv}(\alpha') \subseteq \cC(\mA)$.
    \end{enumerate}
\end{proposition}

Let us briefly unpack these statements above. The first two state that \eqref{eqn:mmi} and trace-\eqref{eqn:mi} are mutually incomparable unless one takes the threshold matrix $\mA$ to be a multiple of $\mI$, in which case \eqref{eqn:mmi} is weaker than trace-\eqref{eqn:mi}. Statements 3 and 4 indicate that \eqref{eqn:mmi} and trace-\eqref{eqn:mi} can simulate each other only by increasing the error probability $\alpha$ {substantially}. Statements 5 and 6 imply respectively that, while one can always simulate quadratic form-\eqref{eqn:mi} with \eqref{eqn:mmi} by \emph{infinitesimally} increasing the error probability $\alpha$, the converse is never true. Proposition~\ref{prop:tight} is proved in \cref{sec:pftight}.

Similar reasoning holds for \eqref{eqn:mvi-new} versus \eqref{eqn:vi} as well because the concentration regions, while holding time-uniformly or at a stopping time, remain in the same forms of $\cC_{\tr}(\alpha)$, $\cC_{\vv}(\alpha)$, and $\cC(\mA)$. We shall spell out the statistical consequences of these linear algebra results in \cref{sec:exp}.

\subsection{Stepwise Scalarization}\label{sec:sclr}



As we mentioned earlier in \cref{sec:examples-mean},
\citet[Section 6.5]{howard2020time} prove, under different conditions, numerous inequalities for $\cS_d$-valued martingale difference $\{ \mZ_n \}$ in the style of
\begin{equation}\label{eqn:howard-general}
   \Exp (\exp( \mZ_n -  \mC_n) | \cF_{n-1}) \matle \exp(  \mC_n'  ),
\end{equation}
for some respectively $\cF_{n}$-, $\cF_{n-1}$-measurable symmetric matrices $\mC_n, \mC_n'$.

At this step, our methods in \cref{sec:mtg} and that of \cite{howard2020time} diverged.
We have been able to directly construct the matrix-valued supermartingale
\begin{equation}\label{eqn:howard-gen-mat}
\exp\left(-\frac{\mC_n'}{2}\right)  \exp\left( \frac{ \mZ_n -  \mC_n}{2}\right) \dots   \exp\left( \frac{ \mZ_n -  \mC_n}{2}\right) \exp\left(-\frac{\mC_n'}{2}\right) 
\end{equation}
via Lemma~\ref{lem:mateval}, leading to a computable sequential test by our new \eqref{eqn:mvi-new}. \cite{howard2020time}, on the other hand, construct \emph{scalar} e-processes from the condition \eqref{eqn:howard-general} by applying a trace scalarization technique. Let us rephrase their ``master lemma" behind this construction.

\begin{lemma}[Lemma 4 in \cite{howard2020time}, rephrased and generalized]\label{lem:howardlieblemma} Let $\{ \mZ_n \}$ be an $\cS_d$-valued martingale difference sequence adapted to $\{ \cF_n \}$. Let $\{\mC_n\}$ be an $\cS_d$-valued adapted process, $\{\mC_n'\}$ be an $\cS_d$-valued predictable process, w.r.t.\ the same filtration. If \eqref{eqn:howard-general} holds for all $n$, then, the process
\begin{equation}\label{eqn:howard-gen-nsm}
    L_n = \tr \exp \left( \sum_{i=1}^n \mZ_i -  \sum_{i=1}^n  ( \mC_i + \mC_i' )  \right)
\end{equation}
is a supermartingale. Further,
\begin{equation}\label{eqn:howard-gen-e-proc}
   L_n \ge \exp \left( \lambda_{\max}\left(\sum_{i=1}^n \mZ_i \right)-   \lambda_{\max}\left( \sum_{i=1}^n ( \mC_i + \mC_i' ) \right)  \right).
\end{equation}
\end{lemma}

The proof of this lemma can be found in \cref{sec:pfhwl}. 
We have numerous remarks regarding this lemma. 

First, the upper bounded expression \eqref{eqn:howard-gen-e-proc} is an instance of e-processes (if divided by $L_0 = d$), introduced earlier in \cref{sec:mat-e-proc}. The larger process
\eqref{eqn:howard-gen-nsm}, a supermartingale, can be used for a sequential test; whereas the smaller process \eqref{eqn:howard-gen-e-proc}, an e-process, leads to closed-form spectral time-uniform concentration inequalities (i.e.\ confidence sequences), but as a test process it is less powerful than \eqref{eqn:howard-gen-nsm}.

Second, the divergence from the common assumption \eqref{eqn:howard-general} to the matrix-valued \eqref{eqn:howard-gen-mat} and scalar-valued \eqref{eqn:howard-gen-nsm} test supermartingale  indicates that for \emph{some} testing problems, two different routes can be pursued: our ``matrix-native'' approach, and the scalarization approach to the test process.
Examples in \cref{sec:examples-mean} that fit in the assumption \eqref{eqn:howard-general} include Examples~\ref{ex:sn} and~\ref{ex:sym} which \citet[Section 6.5]{howard2020time} originally prove, but also Example~\ref{ex:emp-bern-ms}:
To see that Example~\ref{ex:emp-bern-ms} is in the form of \eqref{eqn:howard-general}, 
\begin{equation}
   \exp( \gamma_n(\mX_n - \widehat{\mX}_n)) =  \exp( \underbrace{
 \gamma_n(\mX_n - \mM_n)}_{:= \mZ_n} + \underbrace{\gamma_n(\mM_n - \widehat{\mX}_n)}_{:= \mC_n}  ),
\end{equation}
which is exactly what is done by \citet[Theorem 4.2]{wang2024sharp} to derive a sharp closed-form matrix empirical Bernstein inequality. One of the crucial ways the two approaches differ is that with the scalar-valued test process \eqref{eqn:howard-gen-nsm}, one rejects the null when the process surpasses $d/\alpha$; whereas with the matrix-valued test process \eqref{eqn:howard-gen-mat}, one rejects when the process $\nmatle \mA$ where $\mA$ is a matrix chosen \emph{freely} beforehand, only subject to the ``size constraint'' that $\tr(\mA^{-1}) = \alpha$. Indeed, the extra degrees of freedom that come with $\mA$ play a crucial role as we compare the two approaches empirically with the {finite-variance} self-normalizing example (Example~\ref{ex:sn}) in \cref{sec:exp}.

Third, the assumption \eqref{eqn:howard-general} for scalarization being \emph{more restrictive} than our Lemma~\ref{lem:mateval}, for some testing problems our concept of matrix-valued test supermartingales is unavoidable. These include Example~\ref{ex:bet} and~\ref{ex:cat}, to which the stepwise scalarization Lemma~\ref{lem:howardlieblemma} is not applicable. Their one-step condition $\Exp(\mE_n |\cF_{n-1}) \matle \mF_n$ leads naturally to a matrix-valued test supermartingale, and scalar-valued test supermartingales might be obtained by taking the trace or quadratic form \emph{after} the formation of the matrix-valued one, which is not ideal according to \cref{sec:tight}. 

\subsection{Randomized Ahlswede-Winter Bounds}\label{sec:randfix}

The uniformly randomized Ville's inequality \eqref{eqn:uvi} by \cite{ramdas2023randomized} is obtained by applying the uniformly randomized Markov's inequality \citep[Theorem 1.2]{ramdas2023randomized}
\begin{equation}\label{eqn:umi}
    \Pr(X \ge U \sdiv) \le (\Exp X) a^{-1}
\end{equation}
where $U \sim \uzeroone$ independent from $X$. Here, we state a matrix version of \eqref{eqn:umi}, or a randomized version of the matrix Markov inequality \eqref{eqn:mmi}. 

\begin{theorem}[Uniformly Randomized Matrix Markov's Inequality]\label{thm:ur-matmarkov-meta}
     Let $\mX$ be a random matrix taking values in $\cX \subseteq \cS_d^{+}$, and $U\sim \uzeroone$ independent from $\mX$. For any $\mdiv \in \cS_d^{++}$,
     \begin{equation}\label{eqn:ummi}
         \Pr( \mX  \nmatle   U \mdiv) \le \tr( (\Exp \mX) \mdiv^{-1}).
     \end{equation}
\end{theorem}

It is clear that \eqref{eqn:ummi} generalizes the scalar \eqref{eqn:umi} and tightens \eqref{eqn:mmi}. The randomization factor $U$, we note, can be any random variable that is stochastically larger than $\uzeroone$ (``super-uniform''), that is,  $\Pr(U \le x) \le x$ for all $x \ge 0$. More generally, 
we shall present and prove a stronger version of Theorem~\ref{thm:ur-matmarkov-meta} in \cref{sec:ummi-full}, where the scalar randomization factor $U$ is further replaced by a PSD random matrix whose distribution belongs to a class of ``super-uniform'' measures on $\cS_d^{+}$. The generalized version of Theorem~\ref{thm:ur-matmarkov-meta} in \cref{sec:ummi-full} also reveals the equality condition for \eqref{eqn:mmi}, randomized or not.


Theorem~\ref{thm:ur-matmarkov-meta} implies that various matrix concentration inequalities by \cite{ahlswede2002strong} can be tightened via a randomization factor as well. For example, the Chebyshev inequality \eqref{eqn:mci} for $n$ i.i.d.\ matrices can be randomized as follows.

\begin{corollary}[Uniformly Randomized Matrix Chebyshev's Inequality]\label{cor:umci2}
    Let $\mX_1, \dots, \mX_n$ be random matrices taking values in $\cS_d$ with common mean matrix $\Exp \mX_i = \mM$ and variance matrix upper bound $\Var (\mX_i) \matle \mV$, such that $\Exp \{\mX_i - \mM, \mX_j - \mM\} = 0 $ for any $i\neq j$; 
    $U$ a super-uniform random scalar independent from them all.
    Then, letting $\overline{\mX}_n = \frac{1}{n}(\mX_1 + \dots + \mX_n)$, for any $\mdiv \in \cS_d^{++}$.
     \begin{equation}\label{eqn:umci}
        \Pr( \abs(\overline{\mX}_n - \mM) \nmatle 
\sqrt{U} \mdiv ) \le n^{-1} \tr( \mV \mdiv^{-2} ),
    \end{equation}
\end{corollary}

We omit the proof since it is a straightforward application of Theorem~\ref{thm:ur-matmarkov-meta} to the moment bound Lemma~\ref{lem:contr-var}. We next discuss the implication for \eqref{eqn:mvi-new} as the randomization here raises some subtleties in its two equivalent forms, the time-uniform one and the stopped one.

\begin{theorem}[Uniformly Randomized Matrix Ville's Inequality]\label{thm:matvil1-u}
    Let $\{ \mY_n \}$ be an $\cS_d^+$-valued supermartingale and $\tau$ be a stopping time on $\{ \cF_n \}$. 
    Let $U$ be a super-uniform random scalar
    independent from $\cF_\infty$. Then, for any $\mA \in \cS_d^{++}$,
    \begin{equation}\label{eqn:umvi}
        \Pr( \mY_\tau \nmatle  U \mA) \le \tr( (\Exp \mY_0) \mA^{-1} );
    \end{equation}
    and
    \begin{equation}\label{eqn:umvi-until-stop}
        \Pr( \exists n < \tau, \; \mY_n \nmatle \mA \text{ or } \mY_\tau \nmatle  U \mA) \le \tr( (\Exp \mY_0) \mA^{-1} ).
    \end{equation}
\end{theorem}
\begin{proof}
    First, $\Exp \mY_{\tau} \matle \Exp \mY_0$ due to Part 3 of Theorem~\ref{thm:matrix-optional-stopping}. We apply Theorem~\ref{thm:ur-matmarkov-meta} to obtain the first inequality. For the second inequality, we define the stopping time $\nu := \inf\{ n: \mY_n \nmatle \mA  \}$ and apply \eqref{eqn:umvi} to the stopping time $\nu \wedge \tau$.
\end{proof}

These two inequalities above, we note, generalize respectively the stopped and time-uniform statements of \eqref{eqn:mvi-new}, just as the two inequalities by \citet[Theorem 4.1, Corollary 4.1.1]{ramdas2023randomized} generalize the two forms of the scalar \eqref{eqn:vi}. It is important to note that while it is tempting to hope for the following randomized time-uniform \eqref{eqn:mvi-new}:
\begin{equation}
    \Pr( \exists n, \; \mY_n \nmatle U \mA) \le \tr( (\Exp \mY_0) \mA^{-1} ),
\end{equation}
the temptation must be resisted as it is known to \emph{fail} in the scalar case; see the discussion before Corollary 4.1.1 by \cite{ramdas2023randomized}. Theorem~\ref{thm:matvil1} provides the following improvement over \cref{alg:seqtest} where the user may prespecify a stopping criterion $\tau$ beforehand. As is already stated in \cref{alg:seqtest}, whenever $\mL \mR \nmatle \mA$, $\cH_0$ is rejected; what is now different is that when the stopping criterion $\tau$ is met and $\cH_0$ is not yet rejected, the user can draw a $U \sim \uzeroone$ to see if $\mL \mR \nmatle U\mA$, and reject $\cH_0$ when this happens.

\section{Experiments: To Scalarize or Not to Scalarize}\label{sec:exp}

We now return to our discussion in \cref{sec:sclr} that from the condition
\begin{equation}
    \Exp (\exp( \mZ_n -  \mC_n) | \cF_{n-1}) \matle \exp(  \mC_n'  )
\end{equation}
that holds under the null, one can either use our approach to construct a matrix test supermartingale via Lemma~\ref{lem:mateval}, or use the stepwise scalarization technique formalized as Lemma~\ref{lem:howardlieblemma} to construct a scalar test supermartingale. We now compare these two empirically.


We conduct the comparison with the {finite-variance} self-normalized Example~\ref{ex:sn}, and consider the two corresponding sequential test algorithms, \tm\ and \ts. More specifically, both receive the stream of i.i.d\ matrices $\mX_1,\mX_2,\dots$ in $\cS_d$ with mean $\mM$ and variance upper bound $\mI$ and sequentially test the null hypothesis $H_0: \mM = 0$.
\tm\ computes the matrix-valued process
\begin{multline}
    \mY_n^{\textsf{SN}} =   \exp \left( -\frac{\gamma_1^2 + \dots + \gamma_n^2}{3} \right) \cdot  \exp\left(\frac{\gamma_1 \mX_1}{2} - \frac{\gamma_1^2 \mX_n ^2}{12}\right) \dots     \exp\left(\frac{\gamma_n \mX_n}{2} - \frac{\gamma_n^2 \mX_n ^2}{12}\right)
    \\
     \exp\left(\frac{\gamma_n \mX_n}{2} - \frac{\gamma_n^2 \mX_n ^2}{12}\right)   \dots  \exp\left(\frac{\gamma_1 \mX_1}{2} - \frac{\gamma_1^2 \mX_1 ^2}{12}\right) ,
\end{multline}
rejecting when $ \mY_n^{\textsf{SN}} \nmatle \mA$ for some threshold matrix $\mA \in \cS_d^{++}$ such that $\tr (\mA^{-1}) = \alpha$ picked beforehand; \ts\ computes the scalar-valued process
\begin{equation}
     L_n^{\textsf{SN}} =   \exp \left( -\frac{\gamma_1^2 + \dots + \gamma_n^2}{3} \right) \cdot  \tr \exp \left( \sum_{i=1}^n\gamma_i \mX_i-  \sum_{i=1}^n \frac{\gamma_i^2 \mX_i^2 }{6}    \right),
\end{equation}
rejecting when $ L_n^{\textsf{SN}} \ge d/\alpha$. Note that $ L_n^{\textsf{SN}}$ does \emph{not} equal $\tr (\mY_n^{\textsf{SN}})$ in general, an important point that we shall come to later.

Under the null,  both \tm\ and \ts\ safely control the type 1 error (the probability of falsely rejecting the null) under the predefined level $\alpha$ due to \eqref{eqn:mvi-new} and \eqref{eqn:vi}. For power comparison, we shall first discuss the case when only $n = 1$ observation is revealed, and then
move on to the case with an indefinite sequence of observations, with examples where \tm\ outpowers \ts\ under decreasingly restrictive conditions and finally no restriction at all. {We shall specify the choice of the $\{ \gamma_n \}$ sequence as needed.}

\subsection{Comparison with a Single Observation}

When only $n = 1$ observation is revealed, $  L_1^{\textsf{SN}} = \tr  (\mY_1^{\textsf{SN}})$. We shall see that the comparison depends heavily on the rejection threshold matrix $\mA$ that \tm\ picks beforehand, as rejection happens on the complements of the ``concentration regions'' discussed in \cref{sec:tight}. The following are direct corollaries of Proposition~\ref{prop:tight}.

\begin{proposition}\label{prop:naivematrix}
    Suppose \emtm\ picks $\mA = (d /\alpha )\mI$. Then, if \emtm\ rejects at $n=1$, \emts\ rejects as well. In this case \emts\ is at least as powerful as \emtm. 
\end{proposition}
\begin{proof}
    If $\mA = (d /\alpha )\mI$ and \tm\ rejects, then $\mY_1^{\textsf{SN}} \nmatle (d /\alpha )\mI$, meaning that $\lambda_{\max}(\mY_1^{\textsf{SN}}) > d/\alpha$. Therefore,  $  L_1^{\textsf{SN}} = \tr  (\mY_1^{\textsf{SN}}) \ge \lambda_{\max}(\mY_1^{\textsf{SN}}) > d/\alpha$.
\end{proof}

\begin{proposition}
    Suppose that $d \ge 2$ and that the random matrix $\mY_1^{\mathsf{SN}}$ has full support, meaning it can equal any matrix in $\cS_d^{++}$. Then, 
    \begin{itemize}
        \item  $(d /\alpha )\mI$ is the only matrix that \emtm\ can pick as $\mA \in \cS_d^{++}$ such that if \emtm\ rejects at $n=1$, \emts\ always rejects as well. Therefore, as long as $\mA \neq (d /\alpha )\mI$, \emts\ is not at least as powerful as \emtm;
        \item It is impossible that  if \emts\ rejects at $n=1$, \emtm\ always rejects as well. So \emtm\ is never at least as powerful as \emts. 
    \end{itemize}
   In summary, for general $\mA$ and $n=1$, neither $\emts$ nor $\emtm$ dominates the other.
\end{proposition}

\begin{proof}
    For the first statement, consider the spectral decomposition of the matrix $\mA \neq (d/\alpha)\mI$,
    \begin{equation}
        \mA = \lambda_1 \mP_1 + \dots + \lambda_d \mP_d, \quad \lambda_1 \le \dots \le \lambda_d.
    \end{equation}
    The matrix $\lambda_1  \mP_1$ satisfies $\tr(\lambda_1  \mP_1) = \lambda_1 < d/\alpha$. Applying a small pertubation, $\lambda_1\mP_1 + \varepsilon \mI \nmatle \mA$, $\tr(\lambda_1\mP_1 + \varepsilon \mI) < d/\alpha$. So, if $\mY_1^{\mathsf{SN}}$ takes a value that is close to $\lambda_1\mP_1 + \varepsilon \mI$, \tm\ rejects at $n=1$ but \ts\ does not.

    For the second statement, the matrix $0.9\mA$ satisfies $\tr(0.9\mA) \ge 0.9 d^2/\alpha > d/\alpha$ by Cauchy-Schwarz inequality and  $0.9\mA \matle \mA$. Therefore,  if $\mY_1^{\mathsf{SN}}$ takes a value that is close to $0.9 \mA$, \ts\ rejects at $n=1$ but \tm\ does not.
\end{proof}

Therefore, one may understand the difference between \tm\ and \ts\ as that \tm\ has the freedom to choose the threshold matrix $\mA$ \emph{in anticipation of the potential alternative}; whereas \ts\ represents some kind of Pareto optimal procedure if completely agnostic about the reality. For \tm\ to be more powerful it has to choose $\mA$ wisely according to some prior knowledge. The following example further showcases this point.

\begin{example} Suppose \tm\ knows beforehand what $\mX_1$ is (and hence what $\mY_1^{\mathsf{SN}}$ is), and $\lambda_d > 1/\alpha$ where
\begin{equation}
        \mY_1^{\mathsf{SN}} = \lambda_1 \mP_1 + \dots + \lambda_d \mP_d, \quad \lambda_1 \le \dots \le \lambda_d
    \end{equation}
    is its spectral decomposition.
Then, \tm\ can always make sure to reject by picking
\begin{equation}
    \mA = \varepsilon^{-1}(d-1)\mP_1 + \dots \varepsilon^{-1}(d-1)\mP_{d-1}  + (\alpha - \varepsilon)^{-1}\mP_d
\end{equation}
with $\varepsilon \in (0,\alpha - \lambda_d^{-1})$. The condition $\lambda_d > 1/\alpha$ here is strictly weaker than the condition under which \ts\ rejects: $\lambda_1+\dots + \lambda_d \ge d/\alpha$.
\end{example}
Indeed, as we can see from these examples with $n=1$  (thus we have only compared the matrix \emph{Markov} inequality with the scalar Markov inequality on the trace), for \tm\ to outpower \ts, apart from prior knowledge, it is favorable that \emph{$\lambda_{\max}(\mY_1^{\mathsf{SN}})$ is an outlier in the spectrum of $\mY_1^{\mathsf{SN}}$}, i.e., its \emph{effective rank} $\er(\mY_1^{\mathsf{SN}}) = {\tr(\mY_1^{\mathsf{SN}})}/{\lambda_{\max}(\mY_1^{\mathsf{SN}})}$ is close to 1. This is exactly what we would predict from the discussion in \cref{sec:tight}, where we mention that \eqref{eqn:mmi} is tighter when the rejection threshold matrix $\mA$ (thus also the random matrix itself) has a lower effective rank. In the very next subsection, we shall see that whereas the low-rank condition $\er(\mY_n^{\mathsf{SN}}) \approx 1$ seems artificial and contrived here with $n=1$, it becomes an almost certain phenomenon with large $n$.

\subsection{Comparison with a Sequence of Observations}


Let us consider first an example of a sequence of commuting observations, where \tm's excess power compared to \ts\ increases as time goes on. 

\begin{example}
    Suppose each observation has the spectral decomposition
\begin{equation}
    \mX_n = \lambda_{1n} \mP_1 + \dots + \lambda_{dn} \mP_d
\end{equation}
with common deterministic projection matrices $\mP_1, \dots, \mP_d$.
For each $k$, $ \lambda_{k1}, \lambda_{k2}, \dots \iid \pi_k $ with mean $ \Exp \lambda_{k1} = \lambda_k \ge 0$ 
and variance $ \Var \lambda_{k1} = \sigma_k^2 \le 1$. Then, the variance $\Var (\mX_1) =  \sigma_{1}^2 \mP_1 + \dots + \sigma_{d}^2 \mP_d \matle \mI$.

{We fix a deterministic sequence of positive scalars $\gamma_n = n^{-1/2}$. For each $k$, we define the auxiliary ``component process'' $\{F_{kn}\}_{n \ge 0}$ as the scalar finite-variance self-normalized test process with $\gamma_n = n^{-1/2}$ as the weight sequence when observing data $\lambda_{k1}, \lambda_{k2},\dots$:
\begin{equation}
    F_{kn} =  \exp \left( -\frac{1}{3}\sum_{i=1}^n \gamma_i^2 + \sum_{i=1}^n \gamma_i \lambda_{ki} - \frac{1}{6} \sum_{i=1}^n \gamma_i^2 \lambda_{ki}^2 \right).
\end{equation}}
The matrix-valued process that \tm\ computes {can then be written as}
\begin{equation}
    \mY_n^{\textsf{SN}} 
     = \sum_{k=1}^d  F_{kn} \mP_k,
\end{equation}
and the scalar-valued process that \ts\ computes is 
\begin{equation}
    L_n^{\textsf{SN}} = \sum_{k=1}^d  F_{kn} = \tr( \mY_n^{\textsf{SN}} ).
\end{equation}
{Note that $\{F_{kn} \}$ satisfies
\begin{equation}
    F_{kn} = F_{k(n-1)}\cdot \exp\left( \frac{\lambda_{kn}}{\sqrt{n}} - \frac{2+\lambda_{kn}^2}{6n} \right).
\end{equation}
For large $n$, the first $n^{-1/2}$ term inside the $\exp(\cdot)$ above dominates the second $n^{-1}$ term.
We see that each scalar component process $\{F_{kn}\}$ eventually grows exponentially fast if $\lambda_k = \Exp \lambda_{kn} > 0$, and shrinks to 0 if $\lambda_k = 0$; moreover, the exponential growth rate of $\{F_{kn}\}$ depends polynomially on the effect size $\lambda_k$.} Consequently, as long as there is any mean disparity between the component distributions $\pi_1, \dots, \pi_d$, say $\lambda_1,\dots, \lambda_{d-1}< \lambda_d$, \emph{in the long run the process $\mY_n^{\textsf{SN}}$ will have a low effective rank} with $\er(\mY_n^{\mathsf{SN}}) = {\tr(\mY_n^{\mathsf{SN}})}/{\lambda_{\max}(\mY_n^{\mathsf{SN}})} \to 1$, as the component $F_{dn} \mP_d$ will exponentially outgrow the rest $F_{1n} \mP_1,\dots, F_{(d-1)n} \mP_{d-1}$. In this case if \tm\ correctly anticipates the largest component $\mP_d$, by picking again the threshold matrix
\begin{equation}\label{eqn:threshold-recipe}
     \mA = \varepsilon^{-1}(d-1)\mP_1 + \dots \varepsilon^{-1}(d-1)\mP_{d-1}  + (\alpha - \varepsilon)^{-1}\mP_d
\end{equation}
with small $\varepsilon$ it will outperform \ts: it only takes $F_{dn}$ to surpass $(\alpha - \varepsilon)^{-1} \approx 1/\alpha$ for \tm\ to reject, whereas it takes $L_n^{\textsf{SN}}\approx F_{dn} $ to surpass $d/\alpha$ for \ts\ to reject. Similar reasoning applies as long as \tm\ knows a \emph{subset} of the $d$ components with larger means. {We finally remark that we can lift the assumption $\lambda_k \ge 0$ here, by using the ``hedging'' argument outlined in Remark~\ref{rmk:bet}. Let $\mY_n^{+}$ and $\mY_n^{-}$ be the processes \tm\ would use with $\gamma_n = n^{-1/2}$ and $\gamma_n = - n^{-1/2}$ respectively. Then, if $|\lambda_1|,\dots, |\lambda_{d-1}| < |\lambda_d|$, the spectrum of $(\mY_n^{+} + \mY_n^{-})/2$ will concentrate on the $\mP_d$ direction, too.}
\end{example}

We experiment with the following setup. We pick an orthonormal basis in $d=3$ dimension: $\vu_1 = (1/\sqrt{2}, 1/\sqrt{2}, 0)\trsp$, $\vu_2 = (0,0,1)\trsp$, $\vu_3 = (1/\sqrt{2}, -1/\sqrt{2}, 0)\trsp$, and let $\mP_1 = \vu_1 \vu_1\trsp$, $\mP_2 = \vu_2 \vu_2\trsp$, $\mP_3 = \vu_3 \vu_3\trsp$. The spectral distributions are set as
$\pi_1 = \cN(0.1,1)$, $\pi_2= \cN(-0.2,1)$, and $\pi_3 = \cN(0.3, 1)$. We consider the following three sequential tests: \emph{smart} \tm, one that follows the recipe \eqref{eqn:threshold-recipe} with $\varepsilon = \alpha/20$;
\emph{naive} \tm, one that is agnostic about the real distribution and picks $\mA = (d/\alpha)\mI$; and \ts, all with the parameters $\gamma_n = n^{-1/2}$ and $\alpha = 0.05$. We record the times at which these tests correctly reject the null, as well as the evolution of $\er^{-1}(\mY_n^{\mathsf{SN}}) = {\lambda_{\max}(\mY_n^{\mathsf{SN}})}/{\tr(\mY_n^{\mathsf{SN}})} $. The findings are displayed in \cref{fig:commuting-exps}: smart \tm\ indeed outpowers the other two, whereas naive \tm\ and \ts\ are similar in power, explained by the quick convergence of $\er(\mY_n^{\mathsf{SN}})$ to 1.

\begin{figure}[h]
    \centering
    \includegraphics[width=0.48\linewidth]{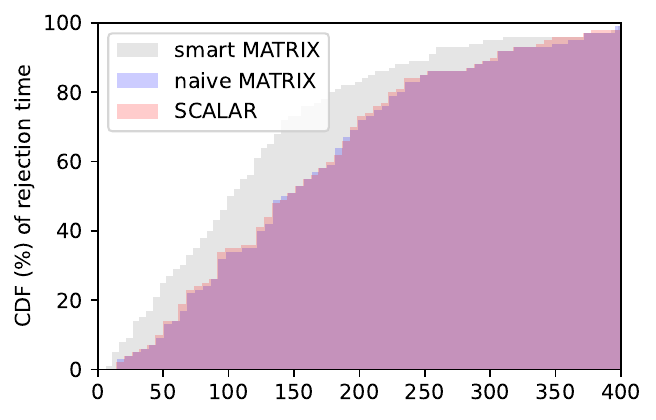}
    \includegraphics[width=0.3\linewidth]{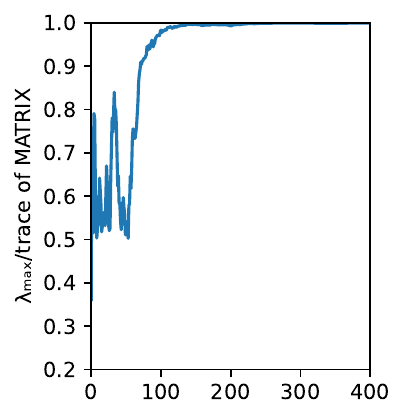}
    \caption{Empirical findings with commuting i.i.d.\ random matrices. \emph{Left.} Comparison of the cumulative distribution functions of rejection times of the sequential tests corresponding to the matrix- and scalar-valued test processes under the same alternative distribution. \emph{Right.} Evolution of the inverse of the effective rank of the matrix-valued test process under the alternative.}
    \label{fig:commuting-exps}
    \vskip 2em
\end{figure}

While this indicates that
\begin{center}
    smart \tm\ $>$ naive \tm\ $\approx$ \ts
\end{center}
in power, we now come to the striking (in light of Proposition~\ref{prop:naivematrix}) phenomenon that with the commutativity assumption removed, it is actually possible that
\begin{center}
 smart \tm\ $>$ naive \tm\ $>$ \ts. 
\end{center}
In the following example that involves non-commuting sequential observations, we consider the application scenario of covariance estimation \citep{ke2019user}.

\begin{example}\label{ex:noncommute}
    Suppose $\mX_n = \vx_n \vx_n\trsp - (2d)^{-1}\mI$ where $\vx_1, \vx_2,\dots \iid \cN(0, \mSg) $. We are to test if the null $\mSg =  (2d)^{-1} \mI$ is true,\footnote{The null covariance  $(2d)^{-1} \mI$ is picked in the mere interest of simplicity so that the variance upper bound $\Var(\mX_n)\matle \mI$ is easily realizable in our example; it is easy to see that any $\mSg \in \cS_d^{++}$ works as well by adjusting the variance upper bound accordingly.} i.e., $ \Exp \mX_n  = 0$, under the 4\textsuperscript{th} moment assumption $\Exp \| \vx_n\|^4 =  (\tr \mSg)^2 +  2 \tr (\mSg^2)   \le 1$, which implies
    \begin{equation}
        \Var(\mX_n) = \Var(\vx_n \vx_n \trsp) \matle \Exp(\vx_n \vx_n \trsp \vx_n \vx_n \trsp) = \Exp(\|\vx_n \|^2 \cdot  \vx_n \vx_n \trsp) \matle \Exp(\|\vx_n \|^2 \cdot  \|\vx_n \|^2 \mI) \matle  \mI.
    \end{equation}
    Suppose in reality \tm\ knows that $\mSg \neq  (2d)^{-1} \mI$ and $\vu_d$ is an eigenvector of $\mSg - (2d)^{-1} \mI$ with large eigenvalue, so that the data $\mX_1, \mX_2,\dots$ will manifest non-zero random signal approximately around this spectral component, it can first complete the orthonormal basis $\vu_1,\dots,\vu_{d-1}$ and picks
    \begin{equation}\label{eqn:threshold-recipe-cov}
     \mA = \varepsilon^{-1}(d-1)\vu_1 \vu_1\trsp + \dots \varepsilon^{-1}(d-1)\vu_{d-1} \vu_{d-1}\trsp  + (\alpha - \varepsilon)^{-1}\vu_d \vu_d\trsp
\end{equation}
to outpower \ts.
    \end{example} 
    Our experiment is set again in the $d=3$ case with  $\vu_1 = (1/\sqrt{2}, 1/\sqrt{2}, 0)\trsp$, $\vu_2 = (0,0,1)\trsp$, $\vu_3 = (1/\sqrt{2}, -1/\sqrt{2}, 0)\trsp$, and the covariance matrix
    \begin{equation}
     \mSg = \left( \begin{matrix}
            1/6  & 1/24 & 0 \\ 1/24 & 1/6 & 0 \\ 0 & 0 & 1/6
        \end{matrix}
        \right) = \text{null hypothesis} + \left( \begin{matrix}
           0 & 1/24 & 0 \\ 1/24 & 0 & 0 \\ 0 & 0 & 0
        \end{matrix}
        \right),
    \end{equation}
    which indeed satisfies $(\tr \mSg)^2 +  2 \tr (\mSg^2) = \frac{61}{144} < 1$.
   The  three sequential tests as before are compared in \cref{fig:nonc-exps}: the \emph{smart} \tm\ that follows \eqref{eqn:threshold-recipe-cov} with $\varepsilon =  \alpha/20$; the \emph{naive} \tm\ with $\mA = (d/\alpha) \mI$; \ts, all with $\gamma_n = n^{-1/2}$ and $\alpha = 0.05$. Besides the unsurprising fact that the smart \tm\ is again the most powerful, rejecting the null earlier than the other two, we also observe the phenomenon that, while in the $n=1$ case, the matrix-valued test process with ``naive" $\mA= (d/\alpha) \mI$ is always less powerful than the scalar-valued test process, this is reversed under non-commutativity with large $n$ as naive \tm\ rejects slightly earlier than \ts\ in our simulation. As also shown in \cref{fig:nonc-exps}, the combination of two auxiliary findings contributes to this phenomenon: $\er(\mY_n^{\mathsf{SN}}) = {\tr(\mY_n^{\mathsf{SN}})}/{\lambda_{\max}(\mY_n^{\mathsf{SN}})} \to 1$ even when observations do not commute; and, while
   $ L_n^{\textsf{SN}} = \tr( \mY_n^{\textsf{SN}} )$ in all our previous examples, now $ \tr( \mY_n^{\textsf{SN}} )$  grows slightly faster than $ L_n^{\textsf{SN}}$. It takes further theoretical investigations to understand why this happens (either in this particular example or in general), but to some extent it hints nonetheless at the possibility that in practice where the effect size is small (requiring large $n$) and matrix observations are non-commuting, there is an inherent advantage using the matrix-valued test process we propose.
\begin{figure}[!h]
    \centering
    \includegraphics[width=0.4\linewidth]{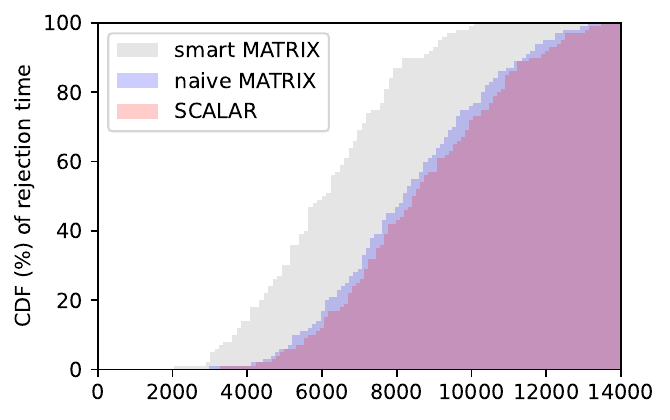}
    \includegraphics[width=0.26\linewidth]{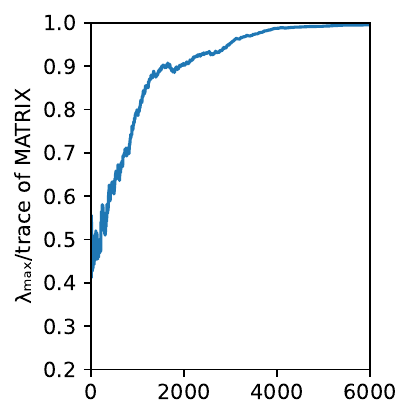}
    \includegraphics[width=0.305\linewidth]{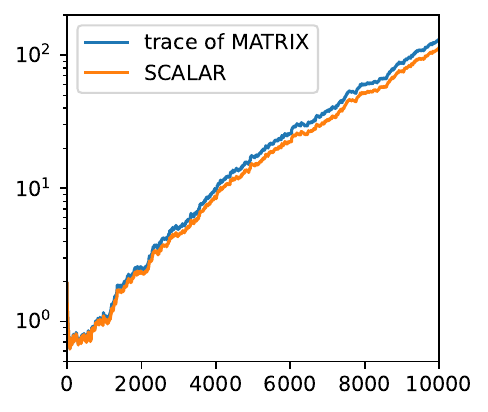}
    \caption{Empirical findings with non-commuting i.i.d.\ random matrices for covariance testing. \emph{Left.} Comparison of the cumulative distribution functions of rejection times of the sequential tests corresponding to the matrix- and scalar-valued test processes under the same alternative distribution. \emph{Middle.} Evolution of the inverse of the effective rank of the matrix-valued test process under the alternative. \emph{Right.} Comparison between the trace of the matrix-valued test process, and the scalar-valued test process under the same alternative distribution.
     }
    \label{fig:nonc-exps}
\end{figure}

%% file: urmat/appendices.tex
\section{On the Randomized Matrix Markov Inequality}\label{sec:ummi-full}

Instead of randomizing the bound with the isotropic diagonal matrix $U\mI$ where $U$ is a uniform or super-uniform random scalar, let us seek an inherently matrix-born definition of uniformity and super-uniformity. 
Scalar super-uniformity has two viable extensions to matrices, one in terms of the trace and the other of the spectral norm.
\begin{definition}
    We call an $\cS_d^{+}$-valued random matrix $\mU$ \emph{spectrally super-uniform} if  $\Pr(\mU 
    \nmatge \mY) \le \lambda_{\max}(\mY)$ for all $\mY \in \cS_d^{+}$, and \emph{$\cY$-spectrally uniform} if, further, equality holds for all $\mY \in \cY \subseteq \cS_d^{+}$. We say $\mU$ is \emph{trace super-uniform}  if  $\Pr(\mU 
    \nmatge \mY) \le \tr \mY$ for all $\mY \in \cS_d^{+}$, and \emph{$\cY$-trace uniform}  if, further, equality holds for all $\mY \in \cY \subseteq \cS_d^{+}$.
\end{definition}
Note that every spectrally super-uniform matrix is a trace super-uniform matrix since $\lambda_{\max} (\mY) \le \tr \mY$. If $U$ is a super-uniform random scalar, then $U\mI$ is a spectrally (hence trace) super-uniform random matrix, because
$
    \Pr( U \mI \nmatge \mY ) = \Pr( U \le \lambda_{\max}(\mY) ) \le \lambda_{\max}(\mY).
$

Straightforwardly the identity matrix $\mI$ is spectrally (hence trace) super-uniform.
If $U \sim \operatorname{Unif}_{(0,1)}$, $U \mI$ is $\cS_d^{[0,1]}$-spectrally uniform, and $\left\{ \vv \vv \trsp :  \|\vv\| \le 1 \right\}$-trace uniform. Having defined trace super-uniformity, we can now state a matrix extension of the uniformly randomized Markov's inequality \eqref{eqn:umi}.

\begin{theorem}[Generalized Uniformly Randomized Matrix Markov Inequality]\label{thm:g-ur-matmarkov-meta}
     Let $\mX$ be a random matrix taking values in $\cX \subseteq \cS_d^{+}$, and $\mU$ a trace super-uniform random matrix taking values in $\cS_d^{+}$ independent from $\mX$. For any $\mdiv \in \cS_d^{++}$,
     \begin{equation}\label{eqn:g-ummi}
         \Pr( \mX  \nmatle  \mdiv^{1/2} \mU \mdiv^{1/2}) \le \tr( (\Exp \mX) \mdiv^{-1}).
     \end{equation}
     The above holds with equality if and only if $\mU$ is $\mdiv^{-1/2} \mathcal{X} \mdiv^{-1/2}$-trace uniform.
\end{theorem}

\begin{proof} By our assumptions,
    \begin{align}
    & \Pr( \mX  \nmatle  \mdiv^{1/2} \mU \mdiv^{1/2}) =\Pr( \mdiv^{-1/2} \mX \mdiv^{-1/2} \nmatle \mU ) = \Exp( \Pr( \mdiv^{-1/2} \mX \mdiv^{-1/2} \nmatle \mU 
 | \mX) )  &
    \\
    \le \ & \Exp( \tr (\mdiv^{-1/2} \mX \mdiv^{-1/2}) ) = \tr( (\Exp \mX) \mdiv^{-1}),
\end{align}
concluding the proof.
\end{proof}

When taking $\mU$ to be the deterministic $\mI$, \eqref{eqn:g-ummi} recovers \eqref{eqn:mmi}.
It is not hard to see that \eqref{eqn:g-ummi} is strictly more general than the \emph{randomized scalar} Markov's inequality \eqref{eqn:umi} and the \emph{scalar-randomized matrix} Markov's inequality \eqref{eqn:ummi}. 
We also see from the monotonicity of the trace and the trace-spectral norm inequality that scalarized trace and spectral bounds are easily obtainable via
\begin{equation}
  \Pr ( \lambda_{\max}(\mX) > \tr( \mdiv \mU) ) \le  \Pr ( \tr \mX > \tr( \mdiv \mU) ) \le  \Pr( \mX  \nmatle  \mdiv^{1/2} \mU \mdiv^{1/2}),
\end{equation}
subject to the same upper bound $\tr( (\Exp \mX) \mdiv^{-1})$.

The following is an example where the inequality of \eqref{eqn:g-ummi} holds with equality. We take $\mU = U \mI$ where $U \sim  \operatorname{Unif}_{(0,1)}$, and $\mdiv \in \cS_d^{++}$. Let $\vx$ be a random vector in the ellipsoid $\{ \vx : \vx\trsp \mdiv^{-1} \vx \le 1 \} \subseteq \mathbb R^d$. Then,
\begin{equation}
 \Pr( \| \mdiv^{-1/2} \vx\|^2 \le U  ) = \Pr( \vx \vx\trsp \nmatle U \mdiv ) = \tr (\mdiv^{-1}  \Exp (\vx \vx\trsp)) =  \Exp \| \mdiv^{-1/2} \vx \|^2,
\end{equation}
where the second equality above is \eqref{eqn:g-ummi} applied on the random matrix $\vx \vx 
 \trsp$. It is not hard to see here the condition for equality is met. Intuitively, this hints that \eqref{eqn:g-ummi} is tighter when the random matrix $\mX$ is approximately of lower rank, a point that we shall return to later in this paper with statistical examples.

Concluding this section, we remark on trace super-uniformity that clearly $U\mI + \mY$ is trace super-uniform when $U\sim \operatorname{Unif}_{(0,1)} $ and $\mY \matge 0$, and this provides an interesting class for our purposes. It may be interesting to fully characterize the set of all trace super-uniform matrices, but we leave it as future work.

\section{Omitted Proofs}
\label{sec:pf}

\subsection{Dominated Integrability Lemma}\label{sec:domint}
It is well-known that if $X$ and $Y$ are scalar random variables such that $c \le X \le Y$ almost surely for some constant $c$ and that $\Exp |Y| < \infty$, it follows that $\Exp |X| < \infty$ as well, and $\Exp X \le \Exp Y$. This type of ``implied integrability" appears frequently in scalar concentration bounds.
Let us prove its symmetric matrix extension for the sake of self-containedness.
\begin{lemma}[Dominated Integrability]\label{lem:domint}
    Let $\mX$ and $\mY$ be $\cS_d^{[c,\infty)}$-valued random matrices for some $c\in \mathbb R$ such that $\mX \matle \mY$ almost surely. Further, suppose $\Exp \mY$ exists. Then, so does $\Exp \mX$ and $\Exp \mX \matle \Exp \mY$. 
\end{lemma}
\begin{proof} Let us prove that each element $X_{ij}$ of the random matrix $\mX$ is integrable.
    Note that for any deterministic $\vv \in \mathbb R^d$, $\vv \trsp \mX \vv \le \vv \trsp \mY \vv$ almost surely. First, taking $\vv = (0,\dots, 0, 1 , 0,\dots 0)\trsp$, we have
    \begin{equation}
        c \le X_{jj} \le Y_{jj} \quad \text{almost surely},
    \end{equation}
    concluding that the diagonal element $X_{jj}$ must be integrable (since $Y_{jj}$ is). Next, taking $\vv = (0,\dots, 0, 1 , 0,\dots ,0, 1 , 0, \dots, 0)\trsp$, we have
    \begin{equation}
       2 c \le 2X_{ij} + X_{ii} + X_{jj} \le 2Y_{ij} + Y_{ii} + Y_{jj} \quad \text{almost surely},
    \end{equation}
    concluding that $2X_{ij} + X_{ii} + X_{jj}$ must be integrable (since $2Y_{ij} + Y_{ii} + Y_{jj}$ is). Therefore, the off-diagonal element $X_{ij}$ is integrable since $X_{ii}$ and $X_{jj}$ are.

    Now that we have established the existence of $\Exp \mX$, it is clear that $\Exp \mX \matle \Exp \mY$ since for any $\vv \in \mathbb R^d$, $
        \vv\trsp (\Exp \mX) \vv = \Exp  (\vv\trsp \mX \vv) \le \Exp  (\vv\trsp \mY \vv) = \vv\trsp (\Exp \mY) \vv$.
\end{proof}

\subsection{Proofs for Forward Supermartingales}
\label{sec:pf-tp}


The following \emph{transfer rule} \citep[Equation 2.2]{tropp2012user} shall commonly be used: suppose $I \subseteq \mathbb R$ and $f, g:I \to \mathbb R$ satisfies $f(x) \le g(x)$, then, since the function $g-f: I \to \mathbb R_{\ge 0}$ identifies canonically with a matrix-to-matrix function $g-f: \cS_d^I \to \cS_d^{+}$, it follows that $f(\mX) \matle g(\mX)$ for any $\mX \in \cS_d^I$.

\begin{proof}[Proof of Lemma~\ref{lem:nsmiff}] {(1) Note that $\mA \matge 0$ if and only if $ \vv \trsp \mA \vv \ge 0$ for all $\vv \in \mathbb Q^d$ due to the continuity of the quadratic form. We have,
    \begin{align}
        & \text{$\{ \mY_n \}$ is a supermartingale}
        \\
        \iff & \Pr\left( \bigcap_n \{ \mY_{n-1} - \Exp(\mY_n|\cF_{n-1}) \matge 0 \}\right)=1
        \\
        \iff & \Pr\left( \bigcap_n \bigcap_{v \in \mathbb Q^d} \{ \vv\trsp(\mY_{n-1} - \Exp(\mY_n|\cF_{n-1}))\vv \ge 0 \}\right)=1
        \\
        \iff &  \Pr\left( \bigcap_n \bigcap_{v \in \mathbb Q^d} \{ \vv\trsp \mY_{n-1} \vv - \Exp(\vv\trsp \mY_n \vv|\cF_{n-1}) \ge 0 \}\right)=1 \\
        \iff & \text{For all $\vv \in \mathbb Q^d$, $\{ \vv\trsp \mY_n \vv \}$ is a supermartingale}
        \\
        \iff & \text{For all $\vv \in \mathbb R^d$, $\{ \vv\trsp \mY_n \vv \}$ is a supermartingale.}  
    \end{align}
    (2) Since $\mA_n = \sum_{i=1}^n  (\Exp(\mY_i|\cF_{i-1}) - \mY_{i-1})$,
    \begin{align}
    & \text{$\{ \mY_n \}$ is a supermartingale}
        \\
        \iff  &  \text{For all $n$, $\Exp(\mY_n|\cF_{n-1}) - \mY_{n-1} \matle 0$}
        \\
        \iff &  \text{For all $n$, $\mA_n - \mA_{n-1} \matle 0$.}   \qedhere
    \end{align}
    
    }
\end{proof}

\begin{proof}[Proof of Example~\ref{ex:sn}]
Let us replicate the key steps of the proof of Lemma 3(f) in \cite{howard2020time}.
    Recalling from \citet[Proposition 12]{delyon2009exponential} that $\exp(x - x^2/6) \le  1 + x + x^2/3 $ for all $x\in \mathbb R$, we have, due to the transfer rule
    \begin{equation}
         \exp(\gamma_n (\mX_n - \mM_n)- \gamma_n^2 (\mX_n - \mM_n)^2/6)  \matle 1 + \gamma_n (\mX_n - \mM_n) + \gamma_n^2 (\mX_n - \mM_n)^2/3 .
    \end{equation}
    This allows us to use Lemma~\ref{lem:domint},
    \begin{align}
       & \Exp\left( \exp(\gamma_n (\mX_n - \mM_n)- \gamma_n^2 (\mX_n - \mM_n)^2/6)   \mid \cF_{n-1} \right) 
        \\
        \matle \ &   \Exp\left( 1 + \gamma_n (\mX_n - \mM_n) + \gamma_n^2 (\mX_n - \mM_n)^2/3   \mid \cF_{n-1} \right)
        \\
        = \ &   1 +  \Exp\left(  \gamma_n^2 (\mX_n - \mM_n)^2/3   \mid \cF_{n-1} \right) \matle 1 + \frac{\gamma_n^2}{3} \mV_n \matle \exp \left( \frac{\gamma_n^2}{3}\mV_n \right),
    \end{align}
    where in the last step we use the transfer rule again with $1+x \le \exp(x)$ for all $x\in \mathbb R$.
\end{proof}

\begin{proof}[Proof of Example~\ref{ex:sym}]
    Let us replicate the key steps of the proof of Lemma 3(d) in \cite{howard2020time}, {another self-normalization-style bound derived this time by the classical symmetrization argument}. Due to the symmetry of $\mX_n-\mM_n$, we introduce a Rademacher random variable $\epsilon$ independent from $\cF_n$,
    \begin{align}
       &\Exp \left(  \exp \left( \gamma_n (\mX_n-\mM_n) - \frac{\gamma_n^2 (\mX_n-\mM_n)^2}{2} \right) \middle | \cF_{n-1}\right)
       \\
       =& \Exp \left(  \exp \left( \gamma_n \epsilon(\mX_n-\mM_n) - \frac{\gamma_n^2 (\mX_n-\mM_n)^2}{2} \right) \middle | \cF_{n-1}\right)
       \\
       = & \Exp \left\{  \Exp \left(  \exp \left( \gamma_n \epsilon(\mX_n-\mM_n) - \frac{\gamma_n^2 (\mX_n-\mM_n)^2}{2} \right) \middle | \cF_{n}\right) \middle | \cF_{n-1} \right\}
       \\
        = & \Exp \left\{ \frac{ \exp \left( \gamma_n (\mX_n-\mM_n) - \frac{\gamma_n^2 (\mX_n-\mM_n)^2}{2} \right) + \exp \left( -  \gamma_n (\mX_n-\mM_n) - \frac{\gamma_n^2 (\mX_n-\mM_n)^2}{2} \right) }{2} \middle | \cF_{n-1} \right\}.
    \end{align}
    Note that the function $f(x) = \frac{\exp(x-x^2/2) + \exp(-x-x^2/2)}{2}$ satisfies $f(x) \le 1$ for all $x\in\mathbb R$, we see from the transfer rule and Lemma~\ref{lem:domint} that 
    \begin{equation}
        \Exp \left(  \exp \left( \gamma_n (\mX_n-\mM_n) - \frac{\gamma_n^2 (\mX_n-\mM_n)^2}{2} \right) \middle | \cF_{n-1}\right) =\Exp \left( f(\mX_n-\mM_n) \middle | \cF_{n-1}\right)  \matle \mI. \qedhere
    \end{equation}
\end{proof}

\begin{proof}[Proof of Example~\ref{ex:cat}] Using the transfer rule on the upper $p$-logarithmic contraction,
    \begin{align}
      &  \Exp\left( \exp(\phi (  \gamma_n(\mX_n - \mM_n) )  )   |\cF_{n-1} \right)
        \\
         \matle &  \Exp\left( \mI + \gamma_n(\mX_n - \mM_n) 
         + (1/p) \gamma_n^p \left( \abs(\mX_n - \mM_n)\right)^p  |\cF_{n-1} \right)
         \\
         \matle & \mI + \frac{\gamma_n^p}{p} \mV_n;
    \end{align}
    and on the lower $p$-logarithmic contraction,
    \begin{align}
      &  \Exp\left( \exp(-\phi (  \gamma_n(\mX_n - \mM_n) )  )   |\cF_{n-1} \right)
        \\
         \matle &  \Exp\left( \mI - \gamma_n(\mX_n - \mM_n) 
         + (1/p) \gamma_n^p \left( \abs(\mX_n - \mM_n)\right)^p  |\cF_{n-1} \right)
         \\
         \matle & \mI + \frac{\gamma_n^p}{p} \mV_n.
    \end{align}
    This concludes the proof.
\end{proof}

\subsection{Proofs for Backward Submartingales}\label{sec:pfbw}

\begin{lemma}[Contraction of the Matrix Variance]\label{lem:contr-var}
    Let $\mX_1, \dots, \mX_n$ be random matrices taking values in $\cS_d$ with common mean matrix $\Exp \mX_i = \mM$ and variance matrix upper bound $\Var (\mX_i) \matle \mV$, such that $\Exp \{\mX_i - \mM, \mX_j - \mM\} = 0 $ for any $i\neq j$.
    Then, for $\overline{\mX}_n = \frac{1}{n}(\mX_1 + \dots + \mX_n)$, we have
   \begin{equation}\label{eqn:mat-contr-variance}
       \Var (\overline{\mX}_n) \matle n^{-1}\mV.
   \end{equation}
\end{lemma}

\begin{proof}
    Note that
    \begin{equation} \small
         \Var (\overline{\mX}_n) = n^{-2}\Exp \left( \sum_{i=1}^n (\mX_i - \mM)   \right)^2 = n^{-2}\left( \sum_{i=1}^n \Var(\mX_i) + \sum_{i < j} \Exp\{\mX_i - \mM, \mX_j - \mM\}   \right) \matle n^{-1}\mV. \qedhere
    \end{equation}
\end{proof}

\begin{proof}[Proof of Theorem~\ref{thm:p-cheb-trace}]
    Due to {Fact}~\ref{fct:tracejen},
\begin{equation}
  \tr \left( \frac 1 n ( (\abs \mX_1)^p + \dots + (\abs \mX_n)^p ) \right) \ge \tr \left( \left(  \abs\frac{\mX_1+\dots + \mX_n}{n} \right)^p \right).
\end{equation}
Therefore,
\begin{equation}
    r_n = \tr \left( \abs(\overline{\mX}_n - \mC)^p \right)
\end{equation}
is a scalar nonnegative backward submartingale. Using the scalar backward Ville's inequality \eqref{eqn:bvi}, 
\begin{equation}
    \Pr\left( \exists n, \;  \tr \left( \abs(\overline{\mX}_n - \mM)^p \right) \ge a^p \right) \le  a^{-p} \Exp r_1 =  a^{-p} \tr \mV_p.
\end{equation}
Finally, noting that $\| \mA \|^p \le \tr((\abs \mA)^p)  $
concludes the proof.
\end{proof}

\subsection{Proof of Proposition~\ref{prop:tight}}\label{sec:pftight}
Recall for convenience the definition of sets
\begin{gather}
    \cC_{\tr}(\alpha) = \{ \mX \in \cS_d^{+} : \tr \mX \le d/\alpha \}, \\
    \cC_{\vv}(\alpha) = \{ \mX \in \cS_d^{+} : \vv\trsp \mX \vv \le 1/\alpha \}, \\
    \cA_{\alpha} = \{ \mA \in \cS_d^{++} :  \tr (\mA^{-1}) = \alpha \}, \\
    \cC({\mA}) = \{ \mX \in \cS_d^{+} :  \mX \matle \mA \}.
\end{gather}

\begin{proof}[Proof of Statement 1]
    Let $\mX \in \cC_{\tr}(\alpha)$. Then $\lambda_{\max}(\mX) \le \tr \mX \le d/\alpha$. Therefore $\mX \in \cC((d/\alpha)\mI)$. To see that there exists an $\mX \in \cC((d/\alpha)\mI)$ such that $\mX \notin \cC_{\tr}(\alpha)$, one can simply take $\mX = (d/\alpha)\mI$.
\end{proof}

\begin{proof}[Proof of Statement 2]
    To see that there exists an $\mX \in \cC(\mA)$ such that $\mX \notin \cC_{\tr}(\alpha)$, one can take $\mX = \mA$ as $\tr(\mA) \ge d^2/\alpha > d/\alpha$. To see that  there exists an $\mX \in \cC_{\tr}(\alpha)$ such that $\mX \notin \cC(\mA)$, suppose $\lambda_{\min}(\mA) = d/\alpha - 2\delta$ where $\delta > 0$. Take $\mX$ to be a matrix with the same {eigenvectors} as $\mA$ and eigenvalues $d/\alpha - \delta, \delta/(d-1), \dots, \delta/(d-1)$.
\end{proof}

\begin{proof}[Proof of Statement 3]
    First, when $\alpha' \ge  \frac{d}{\er(\mA^{-1})} \cdot \alpha $, let $\mX \in \cC_{\tr}(\alpha')$ then $\lambda_{\max}(\mX) \le \tr(\mX) \le d/\alpha' \le 1/\lambda_{\max}(\mA^{-1}) = \lambda_{\min}(\mA)$ so $\mX \in \cC(\mA)$. Second, when $\alpha' < \frac{d}{\er(\mA^{-1})} \cdot \alpha$, let $\delta = (d/\alpha' - \lambda_{\min}(\mA))/2$ and
    take $\mX$ as a matrix with the same {eigenvectors} as $\mA$ and eigenvalues $\lambda_{\min}(\mA)+ \delta, \delta/(d-1), \dots,  \delta/(d-1) $ to see that $\mX \in \cC_{\tr}(\alpha')$ but $\mX \notin \cC(\mA)$.
\end{proof}

\begin{proof}[Proof of Statement 4]
    First, when $\alpha' \le d (\tr \mA)^{-1}$, let $\mX \in \cC(\mA)$ then $\tr \mX \le \tr \mA \le d/\alpha'$ so $\mX \in \cC_{\tr}(\alpha')$. Second, when $\alpha' > d (\tr \mA)^{-1}$, take $\mX = \mA$ to see that $\mX \in \cC(\mA) $ but $\mX \notin \cC_{\tr}(\alpha')$.
\end{proof}

\begin{proof}[Proof of Statement 5] Complete the orthonormal basis $\{\vv_k\}_{1\le k \le d}$ where $\vv_1= \vv$ and let $\mA = \alpha^{-1}\vv \vv\trsp + \sum_{k=2}^d \delta \vv_k \vv_k\trsp$ for $\delta = {\frac{d-1}{\alpha' - \alpha}}$.
\end{proof}

\begin{proof}[Proof of Statement 6] Let $\mA$, $\alpha'$, and $\vv$ be arbitrary. Complete the orthonormal basis as in the previous proof. To find an $\mX \in \cC_{\vv}(\alpha')$ such that $\mX \notin \cC(\mA)$, take $\mX  = {\alpha'}^{-1}\vv \vv\trsp + \sum_{k=2}^d L \vv_k \vv_k\trsp$ for some sufficiently large $L$.
    
\end{proof}

\subsection{Proof of Lemma~\ref{lem:howardlieblemma}}\label{sec:pfhwl}
\begin{proof}
    Due to the monotonicity of $\log$ ({Fact}~\ref{fct:opmonofuns}), the condition \eqref{eqn:howard-general} implies
    \begin{equation}\label{eqn:howard-general-log}
    \log \Exp (\exp( \mZ_n -  \mC_n) | \cF_{n-1}) \matle \mC_n'.
\end{equation}
    Now recall Lieb's concavity theorem \citep{lieb1973convex}: for any $\mH \in \cS_d$, the map $\mX \mapsto \tr \exp(\mH + \log \mX)$ ($\cS_d^{++} \to (0,\infty)$) is concave. Therefore,
    \begin{align}
       \Exp (L_n|\cF_{n-1}) 
        =  \ &  \Exp \left( \tr \exp \left( \sum_{i=1}^{n-1} \mZ_i -  \sum_{i=1}^{n-1}( \mC_i + \mC_i' )  - \mC_n' + \log   \e^{ \mZ_n -  \mC_n  }  \right) \middle| \cF_{n-1} \right)
       \\
       & \text{(Jensen's inequality)} \nonumber
       \\
       \le \ &   \tr \exp \left( \sum_{i=1}^{n-1}\mZ_i -  \sum_{i=1}^{n-1} ( \mC_i + \mC_i' ) -\mC_n' + \log \Exp  (\e^{ \mZ_n - \mC_n   } |\cF_{n-1}) \right) 
       \\
       & \text{(by \eqref{eqn:howard-general-log} and {Fact}~\ref{fct:tracemono})} \nonumber
       \\
       \le \ &  \tr \exp \left( \sum_{i=1}^{n-1} \mZ_i -  \sum_{i=1}^{n-1}  ( \mC_i + \mC_i' )  - \mC_n' + \mC_n'  \right) = L_{n-1},
    \end{align}
    concluding the proof that $\{ L_n \}$ is a supermartingale.
    Finally, observe that
    \begin{align}
        L_n & =  \tr \exp \left( \sum_{i=1}^n \mZ_i -  \sum_{i=1}^n  ( \mC_i + \mC_i' )  \right)
        \\
        & \ge   \tr \exp \left( \sum_{i=1}^n \mZ_i -  \lambda_{\max} \left(\sum_{i=1}^n  ( \mC_i + \mC_i' ) \right) \mI \right)
        \\
        & \ge  \lambda_{\max}   \exp \left( \sum_{i=1}^n \mZ_i -  \lambda_{\max} \left(\sum_{i=1}^n  ( \mC_i + \mC_i' ) \right) \mI \right) 
        \\
        & = \exp \lambda_{\max} \left( \sum_{i=1}^n \mZ_i -  \lambda_{\max} \left(\sum_{i=1}^n  ( \mC_i + \mC_i' ) \right) \mI \right) 
        \\
        & =  \exp \left( \lambda_{\max}\left(\sum_{i=1}^n \mZ_i \right)-   \lambda_{\max}\left( \sum_{i=1}^n  ( \mC_i + \mC_i' ) \right)  \right),
    \end{align}
    concluding the proof.
\end{proof}


\section{Miscellaneous Technical Results}\label{sec:misc}

\subsection{Matrix MGF Bounds}\label{sec:mgf-logmgf}

In \cref{tab:mgf}, we state several results in the random concentration literature on controlling the matrix MGF $\mG(\gamma) = \Exp (\e^{\gamma(\mX_1 - \Exp \mX_1)} )$ for a random matrix $\mX_1$ satisfying certain conditions. If these conditions are met conditionally instead, bounds hold conditionally as well.

\begin{table}[!h]  \small
    \centering
    \begin{tabular}{c|c|c}
        Name & Condition & $\mG(\gamma)$  \\ \hline
         Uni-Gaussian & $\mX_1 - \mM = G_1\mC, G_1 \sim \mathcal{N}_{0,1}$ & $= \e^{\gamma^2 \mC^2/2}$ 
        \\
        Rademacher & $\mX_1 - \mM = R_1\mC, R_1 \sim \operatorname{rad}_{1/2}$ & $\matle \e^{\gamma^2 \mC^2/2}$ 
        \\ 
        Bennett & $\lambda_{\max}(\mX_1 - \mM) \le 1, \Var(\mX_1) = \mV$ & $\matle \e^{ (\e^\gamma - \gamma - 1) \mV}$ 
        \\
        Bernstein & $\Exp (\mX_1 - \mM)^p \matle \frac{p!}{2}\mC^2$ for $p=2,3,\dots$&  $\matle \e^{\frac{\gamma^2}{2(1-\gamma)} \mC^2/2}, \gamma \in (0,1) $
    \end{tabular}
    \caption{Some important matrix MGFs. Here $\mC, \mV$ are deterministic matrices, and $\mM = \Exp \mX_1$.
    }
    \label{tab:mgf}
\end{table}

\begin{proof}[References for \cref{tab:mgf}] The uni-Gaussian MGF bound is proved by \citet[Lemma 4.6.2]{tropp2015introduction} via a Taylor expansion argument.
    For Rademacher, the bound $\mG(\gamma) \matle \exp(\gamma^2 \mC^2 / 2)$ is proved by \citet[Lemma 4.6.3]{tropp2015introduction} via a scalar inequality similar to the Hoeffding's inequality plus the transfer rule (see \cref{sec:pf-tp}).  Bennett  \citep[Lemma 6.7]{tropp2012user} and Bernstein \citep[Lemma 6.8]{tropp2012user}
 follow from a similar reasoning applied to the respective series expansion. 
\end{proof}

\subsection{A Matrix Doob's $L^p$ Inequality}\label{sec:further}

Applying an operator-monotone and concave function to a matrix supermartingale, we get another matrix supermartingale. In particular,
\begin{lemma}
    Let $\{ \mM_n \}$ be an $\cS_d^+$-valued supermartingale on $\{ \cF_n \}$ and $p\in[1,\infty)$. Then, $\{ \mM_n^{1/p} \}$ and $\{ \log \mM_n \}$ are also supermartingales.
\end{lemma}
\begin{proof}
    The function $\mX \mapsto \mX^{1/p}$ is operator-monotone and concave. So
    \begin{equation}
        \Exp( \mM_n^{1/p} | \cF_{n-1} ) \matle \Exp ( \mM_n | \cF_{n-1}) ^{1/p} \matle \mM_{n-1}^{1/p}.
    \end{equation}
    The same proof applies for the logarithm as well.
\end{proof}

{Fact}~\ref{fct:tracemono} and {Fact}~\ref{fct:tracejen}, on the other hand, convert a matrix-valued (super/sub)martingale to a scalar-valued one by trace scalarization. For example, denoting by $\| \mX \|_p$ the Schatten $p$-norm of $\mX \in \cS_d$; that is, $\| \mX \|_p = (\tr ((\abs \mX)^p))^{1/p}$.
\begin{lemma}
    Let $\{ \mM_n \}$ be an $\cS_d^+$-valued submartingale or $\cS_d$-valued martingale, and $p\ge 1$. Then, $\{ \| \mM_n\|^p_p \}$ is a submartingale.
\end{lemma}
\begin{proof} Using {Fact}~\ref{fct:tracejen},
     \begin{equation}
        \Exp( \tr (\abs \mM_n)^p | \cF_{n-1} ) = \tr  \Exp ( (\abs \mM_n)^p | \cF_{n-1} )   \ge \tr (\abs \Exp (\mM_n | \cF_{n-1})) ^{p}.
    \end{equation}
    If $\{ \mM_n \}$ is a martingale, the right hand side above is $ \tr (\abs \mM_{n-1} ) ^{p}$. If $\mM_n \matge 0$ is a submartingale, then we have
     by {Fact}~\ref{fct:tracemono},
    \begin{equation}
        \tr (\abs \Exp (\mM_n | \cF_{n-1})) ^{p}  \ge   \tr (\abs \mM_{n-1} ) ^{p},  
    \end{equation}
    since the function $x \mapsto |x|^p$ is monotone on $x \ge 0$. 
\end{proof}

In particular, when $p=1$ in above, the nuclear norm process $\{\| \mM_n \|_1\}$ is a submartingale. We now have the following generalization of Doob's $L^p$ maximal inequality (e.g.\ Theorem 11.2 in \cite{klenke2013probability}).
 
\begin{proposition}[Doob-type Nuclear $L^p$ Inequality]
    Let $\{ \mM_n \}$ be an $\cS_d^+$-valued submartingale or $\cS_d$-valued martingale, and $p > 1$. Then,
    \begin{equation}\scriptsize
   \Exp\left(\sup_{0\le n \le N}  \| \mM_n \|^p \right) \le \Exp \left(\sup_{0\le n \le N}   \| \mM_n \|_1^p \right) \le \left( \frac{p}{p-1} \right)^p \Exp \|\mM_N\|^p_1 =  \left( \frac{p}{p-1} \right)^p \Exp  (\|\mM_N\| \cdot \er(\mM_N))^p.
    \end{equation}
\end{proposition}
\begin{proof}
The second inequality is directly from applying Doob's $L^p$ inequality on the nonnegative submartingale $\{ \| \mM_n\|_1 \}$. The first $\le$ and the final $=$ follow from straightward relations between the spectral and nuclear norms, and the effective rank.
\end{proof}


\subsection{Non-symmetric, Non-square Matrices}
\label{sec:dilation}

We finally remark on the extension of our methods to non-symmetric and possibly non-square matrices. Denote by $\cM_{d_1\times d_2}(\mathbb R)$ the space of all $d_1\times d_2$ real-valued matrices, and we consider the random observations $\mX_1, \mX_2,\dots$ now as taking values in $\cM_{d_1\times d_2}(\mathbb R)$. We can employ the classical idea of \emph{dilations} (see e.g.\ \cite[Section 2.1.16]{tropp2015introduction}).
Consider the linear map
\begin{equation}
    D: \cM_{d_1\times d_2}(\mathbb R) \to \cS_{d_1+d_2}, \quad D(\mX) = \left( \begin{matrix}
        0 & \mX \\ \mX\trsp & 0
    \end{matrix} \right),
\end{equation}
and we can apply our previously developed bounds or tests on $\{D(\mX_n)\}$, which shall in turn provide concentration inequalities for $\{\mX_n\}$ themselves: either via the relation $\| D(\mX) \| = \| \mX \|$ with the linearity of $D$ to obtain spectral bounds for $\mX_1,\mX_2,\dots$, or one may directly use the resulting (scalar or matrix) e-process as the test process.

In \cref{tab:dilation}, we see that quadratic assumptions on the possibly non-square random matrix $\mX_1$ translate directly to those on the symmetric random matrix $D(\mX_1)$. Conditional distribution assumptions are translated similarly.

\begin{table}[!h]
    \centering
    \begin{tabular}{c|c}
     Assumption on $\mX_1$ &  Assumption on $D(\mX_1)$    \\ \hline \hline
       $\mX_1\trsp \mX_1 \matle \mB_1$ and $\mX_1 \mX_1\trsp \matle \mB_2$  & $D^2(\mX_1) \matle  \left( \begin{matrix}
        \mB_1 &  \\  & \mB_2
    \end{matrix} \right)$ and thus \\ &   $ \left( \begin{matrix}
        -\mB_1^{1/2} &  \\  & -\mB_2^{1/2}
    \end{matrix} \right) \matle D(\mX_1) \matle  \left( \begin{matrix}
        \mB_1^{1/2} &  \\   & \mB_2^{1/2}
    \end{matrix} \right)$   \\ \hline
    $\| \mX_1 \| \le b$ & $\|D(\mX_1)\| \le b$
    \\ \hline
        $\Cov (\mX_1 \trsp) \matle \mV_1$ and $\Cov (\mX_1)  \matle \mV_2$  & $\Var (D(\mX_1)) \matle  \left( \begin{matrix}
        \mV_1 &  \\  & \mV_2
    \end{matrix} \right)$  \\ 
    \end{tabular}
    \caption{Distribution assumptions on $\{\mX_n\}$ imply those on $\{D(\mX_n)\}$. Here, the covariance of a random matrix is defined as $\Cov(\mX) = \Exp   (\mX - \Exp \mX)  (\mX - \Exp \mX)\trsp$.}
    \label{tab:dilation}
\end{table}

For example, the self-normalized matrix supermartingale  (Example~\ref{ex:sn}) can be generalized as follows.

\begin{example}
      Let $\{\mX_n\}$ be $\cM_{d_1\times d_2}(\mathbb R)$-valued random matrices adapted to $\{ 
\cF_n \}$ with  conditional means $\Exp (\mX_n | \cF_{n-1}) = \mM_n$ and conditional covariance upper bounds \newline $\Cov( \mX_n\trsp | \cF_{n-1} ) \matle \mV_{1n} $,  $\Cov( \mX_n | \cF_{n-1} ) \matle \mV_{2n} $ that are predictable. For any predictable positive scalars $\{ \gamma_n \}$,
let
\begin{multline} \scriptsize
    \mE_n =  \left( \begin{matrix}
       - \gamma_n^2 (\mX_n - \mM_n)(\mX_n - \mM_n)\trsp/6   & \gamma_n (\mX_n - \mM_n)  \\  \gamma_n (\mX_n - \mM_n)\trsp &   - \gamma_n^2 (\mX_n - \mM_n)\trsp(\mX_n - \mM_n)/6 
    \end{matrix} \right), 
    \; \mF_n =  -\frac{\gamma_n^2}{3} \left( \begin{matrix}
        \mV_{1n}  &  \\  & \mV_{2n} 
    \end{matrix} \right),
\end{multline}
the process
 \begin{equation}
        \mY_n = \e^{\mF_{1}/2} \e^{\mE_{1}/2} \e^{\mF_{2}/2} \e^{\mE_{2}/2}\dots  \e^{\mF_{n}/2} \e^{\mE_{n}/2}  \e^{\mE_n/2}  \e^{\mF_{n}/2} \dots \e^{\mE_{2}/2}  \e^{\mF_{2}/2} \e^{\mE_{1}/2}  \e^{\mF_{1}/2}
    \end{equation}
    is an $\cS_{d_1+d_2}^{+}$-valued supermartingale.
\end{example}

%% file: urmat/conclusion.tex
\section{Conclusion}

In this paper, we study positive semidefinite matrix supermartingales and backward submartingales. In particular, we present several maximal inequalities that generalize existing results in the well-known scalar case. We demonstrate the statistical applications of these results by noting that for data and hypotheses that inherently take matrix values, our ``matrix-native'' approach is beneficial and can outperform previous methods that often require ``scalarization''.

In particular, sequential testing with matrix-valued test supermartingales and e-processes allows a freely chosen rejection threshold matrix $\mA$ subject to the constraint $\tr(\mA^{-1}) = \alpha$, one that can be chosen according to the prior knowledge to boost power without sacrificing the type 1 error guarantee. On the other hand, taking the threshold matrix $\mA$ to be a multiple of the identity matrix, a degenerate choice that it appears, can still lead to a more powerful test in the large sample regime compared to scalarization. This is an empirical finding of our paper that calls for further understanding.

Our study shall lead to more powerful and flexible sequential hypothesis testing procedures for matrix inference problems including covariance and graph testing.

%% file: urmat/ack.tex
The authors thank Joel A.\ Tropp, Lester Mackey, Tudor Manole, and Reihaneh Malekian for suggestions and discussions, and acknowledge support from NSF grant  DMS-2310718.